\documentclass[11pt,oneside]{amsart}%
\usepackage{amssymb,amsmath,amsfonts,latexsym,amsthm,geometry,graphicx}
\usepackage{amsmath}
\usepackage{amsfonts}
\usepackage{color}
\usepackage{amssymb}
\usepackage{graphicx}
\usepackage[utf8]{inputenc}%
\providecommand{\U}[1]{\protect\rule{.1in}{.1in}}
\newtheorem{theorem}{Theorem}[section]
\newtheorem{corollary}[theorem]{Corollary}
\newtheorem{proposition}[theorem]{Proposition}
\newtheorem{lemma}[theorem]{Lemma}

\theoremstyle{definition}

\begin{document}
\title{Geometry of the closed unit ball of the space of bilinear forms on
$\ell_{\infty}^{2}$}
\author[W. Cavalcante]{Wasthenny Cavalcante}
\address{Departamento de Matem\'{a}tica - Federal University of Pernambuco - Recife - Brazil}
\email{wasthenny.wc@gmail.com and wasthenny@dmat.ufpe.br}
\author[D. Pellegrino]{Daniel Pellegrino}
\address{Departamento de Matem\'{a}tica \\
Universidade Federal da Para\'{\i}ba \\
58.051-900 - Jo\~{a}o Pessoa, Brazil.}
\email{pellegrino@pq.cnpq.br and dmpellegrino@gmail.com}
\thanks{2010 Mathematics Subject Classification: 46G25}
\thanks{Key words: Bilinear forms, Banach spaces, Littlewood's $4/3$ inequality}
\thanks{W. Cavalcante is supported by Capes and D. Pellegrino is supported by CNPq}

\begin{abstract}
We obtain all extreme and exposed points of the closed unit ball of the space
of bilinear forms $T:\ell_{\infty}^{2}\times\ell_{\infty}^{2}\rightarrow
\mathbb{R}.$ We also show that any (norm one) bilinear form $T:\ell_{\infty
}^{2}\times\ell_{\infty}^{2}\rightarrow\mathbb{R}$ for which the optimal
constant of the Littlewood's $4/3$ inequality is achieved is necessarily an
extreme point. In the case of complex scalars we combine analytical and
numerical evidences supporting that, at least for complex bilinear forms with
real coefficients, the optimal constant of Littlewood's $4/3$ inequality seems
to the the trivial, i.e., $1.$

\end{abstract}
\maketitle
%\tableofcontents

\section{Introduction}

Given a Banach space $E$ and a convex set $A\subset E$, a vector $x\in A$ is
called an extreme point of $A$ if $y,z\in A$ with $x=\frac{y+z}{2}$ implies
$y=z.$ The characterization of extreme points of certain Banach spaces is a
fruitful subject of investigation (see \cite{choi1, choi2, choi3, grecu, kim}
and references therein) and the identification of extreme points of the closed
unit ball of certain spaces of polynomials and bilinear forms has been quite
useful in certain optimization problems (see, for instance, \cite{caval}). For
a detailed exposition of the subject we refer to \cite{no}. In the present
paper we present all extreme and exposed points of the closed unit ball of
$\mathcal{L}(^{2}\ell_{\infty}^{2}(\mathbb{R}))$, i.e., the space of all
bilinear forms $T:\ell_{\infty}^{2}\times\ell_{\infty}^{2}\rightarrow
\mathbb{R}$. We also investigate the optimization problem in the closed unit
ball $B_{\mathcal{L}(^{2}\ell_{\infty}^{2}(\mathbb{R}))}$ of $\mathcal{L}%
(^{2}\ell_{\infty}^{2}(\mathbb{R}))$ associated to the best constant of
Littlewood's $4/3$ inequality. More precisely we obtain all bilinear forms in
the closed unit ball of $\mathcal{L}(^{2}\ell_{\infty}^{2}(\mathbb{R}))$ such
that the optimal constant of Littlewood's $4/3$ inequality is achieved.

The paper is organized as follows. In Section 2 we obtain expressions for the
norms of bilinear forms $T:\ell_{\infty}^{2}\times\ell_{\infty}^{2}%
\rightarrow\mathbb{R}$ and $T:\ell_{\infty}^{2}\times\ell_{\infty}%
^{2}\rightarrow\mathbb{C}$ with real coefficients. In Section 3 these results
are used to classify the extreme and exposed points of the closed unit ball of
$\mathcal{L}(^{2}\ell_{\infty}^{2}(\mathbb{R}))$. In Section 4 we use the
results of Section 2 to present all bilinear forms $T:\ell_{\infty}^{2}%
\times\ell_{\infty}^{2}\rightarrow\mathbb{R}$ for which the optimal constant
of the Littlewood's $4/3$ inequality is achieved. More precisely, we obtain
all bilinear forms satisfying the following optimization problem:%

\[
\inf\left\{  \left(  \sum_{j_{1},j_{2}=1}^{2}\left\vert T(e_{j_{1}},e_{j_{2}%
})\right\vert ^{\frac{4}{3}}\right)  ^{\frac{3}{4}}\!\!\!\!\!,\text{ among all
}2\text{--linear forms }T\in B_{\mathcal{L}(^{2}\ell_{\infty}^{2}%
(\mathbb{R}))}\right\}  =\sqrt{2}.
\]
In Section 5 we consider the complex version of this problem and, combining
analytical and numerical approaches, we obtain strong evidence supporting
that, at least for complex bilinear forms with real coefficients, the optimal
constant of the Littlewood's $4/3$ inequality seems to the the trivial, i.e.,
$1.$

\section{Expressions for the norms of bilinear forms on $\mathcal{L}(^{2}%
\ell_{\infty}^{2}(\mathbb{K}))$\label{norms}}

A first step to determine the geometry of the unit ball of $\mathcal{L}%
(^{2}\ell_{\infty}^{2}(\mathbb{K}))$ is to find expressions for the norms.
This is not a very pleasant task, mainly in the case of complex scalars.

\begin{proposition}
\label{normcomp} Let $T:c_{0}\times c_{0}\rightarrow\mathbb{C}$ be given by
$T(z,w)=\sum_{i,j=1}^{2}a_{ij}z_{i}w_{j}$ with $a_{ij}\in\mathbb{R}$. Then

(A)
\[
\left\Vert T\right\Vert =\max\left\{
\begin{array}
[c]{c}%
|a_{11}+a_{21}|+|a_{12}+a_{22}|,|a_{11}-a_{21}|+|a_{12}-a_{22}|,\\
\sqrt{a_{11}^{2}+a_{21}^{2}+2a_{11}a_{21}\frac{\frac{a_{11}^{2}a_{21}^{2}%
}{a_{12}^{2}a_{22}^{2}}(a_{12}^{2}+a_{22}^{2})-(a_{11}^{2}+a_{21}^{2}%
)}{2a_{11}a_{21}(1-\frac{a_{11}a_{21}}{a_{12}a_{22}})}}\\
+\sqrt{a_{12}^{2}+a_{22}^{2}+2a_{12}a_{22}\frac{\frac{a_{11}^{2}a_{21}^{2}%
}{a_{12}^{2}a_{22}^{2}}(a_{12}^{2}+a_{22}^{2})-(a_{11}^{2}+a_{21}^{2}%
)}{2a_{11}a_{21}(1-\frac{a_{11}a_{21}}{a_{12}a_{22}})}}%
\end{array}
\right\}
\]
if $\left(  a_{11},a_{21},a_{12},a_{22}\right)  \in\left(  \mathbb{R}%
\setminus\{0\}\right)  ^{4}$ and $sgn\left(  \frac{a_{11}a_{21}}{a_{12}a_{22}%
}\right)  =-1$ and%

\[
\left\vert \frac{a_{11}^{2}a_{21}^{2}}{a_{12}^{2}a_{22}^{2}}\left(  a_{12}%
^{2}+a_{22}^{2}\right)  -\left(  a_{11}^{2}+a_{21}^{2}\right)  \right\vert
\leq\left\vert 2a_{11}a_{21}\left(  1-\frac{a_{11}a_{21}}{ a_{12}a_{22}%
}\right)  \right\vert ;
\]
(B)
\[
\left\Vert T\right\Vert =\max\left\{  |a_{11}+a_{21}|+|a_{12}+a_{22}%
|,|a_{11}-a_{21}|+|a_{12}-a_{22}|\right\}
\]
otherwise.
\end{proposition}

\begin{proof}
Note that
\[
\left\Vert T\right\Vert =\sup\{\left\Vert T_{z}\right\Vert :\Vert
z\Vert_{\infty}=1\},
\]
where $T_{z}:\ell_{\infty}^{2}(\mathbb{C})\rightarrow\mathbb{C}$ is given by
\[
T_{z}(w)=(a_{11}z_{1}+a_{21}z_{2})w_{1}+(a_{12}z_{1}+a_{22}z_{2})w_{2}.
\]
We thus have
\[
\Vert T\Vert=\sup\left\{  \Vert T_{z}\Vert:\Vert z\Vert_{\infty}=1\right\}
=\sup\left\{  \left\vert a_{11}z_{1}+a_{21}z_{2}\right\vert +\left\vert
a_{12}z_{1}+a_{22}z_{2}\right\vert :\Vert z\Vert_{\infty}=1\right\}  .
\]
Hence, calculating $\Vert T\Vert$ is the same of maximizing the function
\[
f(z)=|a_{11}z_{1}+a_{21}z_{2}|+|a_{12}z_{1}+a_{22}z_{2}|
\]
with the restriction $\Vert z\Vert_{\infty}=1.$ Denoting $z_{j}=x_{j}+iy_{j}$,
$j=1,2$, we have
\begin{align*}
f(z)  &  =\sqrt{\left(  a_{11}x_{1}+a_{21}x_{2}\right)  ^{2}+\left(
a_{11}y_{1}+a_{21}y_{2}\right)  ^{2}}\\
&  +\sqrt{\left(  a_{12}x_{1}+a_{22}x_{2}\right)  ^{2}+\left(  a_{12}%
y_{1}+a_{22}y_{2}\right)  ^{2}}%
\end{align*}

Since $\left\Vert z\right\Vert _{\infty}=1$, we can write $z_{j}=\cos
\theta_{j}+i\sin\theta_{j}$, $j=1,2$. Hence
\[
f(\theta_{1},\theta_{2})=\sqrt{a_{11}^{2}+a_{21}^{2}+2a_{11}a_{21}\cos
(\theta_{1}-\theta_{2})}+\sqrt{a_{12}^{2}+a_{22}^{2}+2a_{12}a_{22}\cos
(\theta_{1}-\theta_{2})}.
\]
By making $t=\theta_{1}-\theta_{2}$ we have
\[
f(t)=\sqrt{a_{11}^{2}+a_{21}^{2}+2a_{11}a_{21}\cos t}+\sqrt{ a_{12}^{2}%
+a_{22}^{2}+2a_{12}a_{22}\cos t}.
\]

\begin{itemize}
\item Proof of (A):

We divide the proof of (A) in two cases:
\end{itemize}

$\circ$ First case. Suppose that $\left(  a_{11},a_{21},a_{12},a_{22}\right)
\in\left(  \mathbb{R}\setminus\{0\}\right)  ^{4}$ and $a_{11}\neq\pm a_{21}$
and $a_{12}\neq\pm a_{22}.$

In this case, since $\left(  a_{11},a_{21},a_{12},a_{22}\right)  \in\left(
\mathbb{R}\setminus\{0\}\right)  ^{4}$ and $a_{11}\neq\pm a_{21}$ and
$a_{12}\neq\pm a_{22}$, $f^{\prime}$ always exists and
\[
f^{\prime}(t)=\frac{-a_{11}a_{21}\sin t}{\sqrt{ a_{11}^{2}+a_{21}^{2}%
+2a_{11}a_{21}\cos t}}+\frac{-a_{12}a_{22}\sin t}{ \sqrt{a_{12}^{2}+a_{22}%
^{2}+2a_{12}a_{22}\cos t}}=0
\]
if and only if $t=k\pi,k\in\mathbb{Z}$ or
\begin{equation}
\frac{-a_{11}a_{21}}{a_{12}a_{22}}=\frac{\sqrt{ a_{11}^{2}+a_{21}^{2}%
+2a_{11}a_{21}\cos t}}{\sqrt{ a_{12}^{2}+a_{22}^{2}+2a_{12}a_{22}\cos t}}.
\label{777}%
\end{equation}
Since $sgn\left(  \frac{a_{11}a_{21}}{a_{12}a_{22}}\right)  =-1$, we have
\begin{equation}
2a_{11}a_{21}\left(  1-\frac{a_{11}a_{21}}{a_{12}a_{22}}\right)  \cos t=\frac{
a_{11}^{2}a_{21}^{2}}{a_{12}^{2}a_{22}^{2}} (a_{12}^{2}+a_{22}^{2}%
)-(a_{11}^{2}+a_{21}^{2}) \label{888}%
\end{equation}

and since
\[
\left\vert \frac{a_{11}^{2}a_{21}^{2}}{a_{12}^{2}a_{22}^{2}}(a_{12}^{2}%
+a_{22}^{2})-(a_{11}^{2}+a_{21}^{2})\right\vert \leq\left\vert 2a_{11}%
a_{21}\left(  1-\frac{a_{11}a_{21}}{a_{12}a_{22}}\right)  \right\vert ,
\]
there is $t_{0}$ such that
\[
\cos t_{0}=\frac{\frac{a_{11}^{2}a_{21}^{2}}{a_{12}^{2}a_{22}^{2}}(a_{12}%
^{2}+a_{22}^{2})-(a_{11}^{2}+a_{21}^{2})}{2a_{11}a_{21}(1-\frac{a_{11}a_{21}%
}{a_{12}a_{22}})}.
\]
Thus
\begin{equation}
\left\Vert T\right\Vert =\max f=\max\left\{
\begin{array}
[c]{c}%
|a_{11}+a_{21}|+|a_{12}+a_{22}|,|a_{11}-a_{21}|+|a_{12}-a_{22}|,\\
\sqrt{a_{11}^{2}+a_{21}^{2}+2a_{11}a_{21}\frac{\frac{a_{11}^{2}a_{21}^{2}%
}{a_{12}^{2}a_{22}^{2}}(a_{12}^{2}+a_{22}^{2})-(a_{11}^{2}+a_{21}^{2}%
)}{2a_{11}a_{21}(1-\frac{a_{11}a_{21}}{a_{12}a_{22}})}}\\
+\sqrt{a_{12}^{2}+a_{22}^{2}+2a_{12}a_{22}\frac{\frac{a_{11}^{2}a_{21}^{2}%
}{a_{12}^{2}a_{22}^{2}}(a_{12}^{2}+a_{22}^{2})-(a_{11}^{2}+a_{21}^{2}%
)}{2a_{11}a_{21}(1-\frac{a_{11}a_{21}}{a_{12}a_{22}})}}%
\end{array}
\right\}  . \label{ji}%
\end{equation}
$\circ$ Second case. Suppose that $\left(  a_{11},a_{21},a_{12},a_{22}\right)
\in\left(  \mathbb{R}\setminus\{0\}\right)  ^{4}$ and $\left(  a_{11}=\pm
a_{21}\text{ or }a_{12}=\pm a_{22}\right)  .$

In this case there are real numbers $t_{0}$ such that $f^{\prime}(t_{0})$ does
not exist. For these values of $t_{0}$ we can see that
\[
f(t_{0})=|a_{11}+a_{21}|+|a_{12}+a_{22}|
\]
or
\[
f(t_{0})=|a_{11}-a_{21}|+|a_{12}-a_{22}|.
\]
For the values of $t$ such that $f^{\prime}(t)$ exists, we proceed as in the
first case; therefore we also obtain (\ref{ji}).

\begin{itemize}
\item \bigskip Proof of (B).
\end{itemize}

We consider three cases:

$\circ$ Case 1. Suppose $\left(  a_{11},a_{21},a_{12},a_{22}\right)
\in\left(  \mathbb{R}\setminus\{0\}\right)  ^{4},$ $a_{11}\neq\pm a_{21}$ and
$a_{12}\neq\pm a_{22}$ with
\[
sgn\left(  \frac{a_{11}a_{21}}{a_{12}a_{22}}\right)  =1.
\]

From (\ref{777}) we can observe that $f^{\prime}(t)=0$ if and only if
$t=k\pi,k\in\mathbb{Z}$ and thus
\[
\left\Vert T\right\Vert =\max f=\max\{|a_{11}+a_{21}|+|a_{12}+a_{22}%
|,|a_{11}-a_{21}|+|a_{12}-a_{22}|\}.
\]

$\circ$ Case 2. Suppose $\left(  a_{11},a_{21},a_{12},a_{22}\right)
\in\left(  \mathbb{R}\setminus\{0\}\right)  ^{4}$, $a_{11}\neq\pm a_{21}$,
$a_{12}\neq\pm a_{22}$,
\[
sgn\left(  \frac{a_{11}a_{21}}{a_{12}a_{22}}\right)  =-1
\]
and%

\[
\left\vert \frac{a_{11}^{2}a_{21}^{2}}{a_{12}^{2}a_{22}^{2}}\left(  a_{12}%
^{2}+a_{22}^{2}\right)  -\left(  a_{11}^{2}+a_{21}^{2}\right)  \right\vert
>\left\vert 2a_{11}a_{21}\left(  1-\frac{a_{11}a_{21}}{a_{12}a_{22}}\right)
\right\vert .
\]
In this case, from (\ref{888}) we also know that $f^{\prime}(t)=0$ if and only
if $t=k\pi,k\in\mathbb{Z}$; therefore
\[
\left\Vert T\right\Vert =\max f=\max\{|a_{11}+a_{21}|+|a_{12}+a_{22}%
|,|a_{11}-a_{21}|+|a_{12}-a_{22}|\}.
\]

$\circ$ Case 3. We may have one of the following situations:

(1) $a_{11}a_{21}=0$ and $a_{12}a_{22}=0$;

(2) $a_{11}a_{21}=0$ and $a_{12}a_{22}\neq0$;

(3) $a_{11}a_{21}\neq0$ and $a_{12}a_{22}=0$;

(4) $a_{11}a_{21}\neq0$ and $a_{12}a_{22}\neq0.$

\bigskip

If we consider (1), $f$ can be written as one of the following expressions:

(a) $f(t) = |a_{11}|+|a_{12}|$;

(b) $f(t) = |a_{11}|+|a_{22}|$;

(c) $f(t) = |a_{21}|+|a_{12}|$;

(d) $f(t) = |a_{21}|+|a_{22}|$.

We thus can write, in any case,
\[
f(t)=|a_{11}+a_{21}|+|a_{12}+a_{22}|=|a_{11}-a_{21}|+|a_{12}-a_{22}|\vspace{
0.2cm}%
\]
and, of course, we obtain the expression of (B).

If we consider (2) there is no loss of generality in supposing $a_{11}=0.$ So,
we get
\[
f(t)=|a_{21}|+\sqrt{a_{12}^{2}+a_{22}^{2}+2a_{12}a_{22}\cos t}%
\]
and we consider two subcases.

$\circ$ Subcase 1. If $a_{12}=a_{22}$ or $a_{12}=-a_{22}$, then there is a
$t_{0}\in\mathbb{R}$ such that $f^{\prime}(t_{0})$ does not exist. In this
case, it is plain that
\[
f(t_{0})=|a_{21}|\leq|a_{11}+a_{21}|+|a_{12}+a_{22}|.
\]
For other values of $t$ we have
\[
f^{\prime}(t)=\frac{-a_{12}a_{22}\sin t}{\sqrt{ a_{12}^{2}+a_{22}^{2}%
+2a_{12}a_{22}\cos t}}%
\]
and thus $f^{\prime}(t_{1})=0$ if and only if $t_{1}=k\pi$ and $k\in
\mathbb{Z}$. For these values of $t_{1}$ we have
\[
f(t_{1})=|a_{21}|+\left\vert a_{12}+a_{22}\right\vert =|a_{11}+a_{21}%
|+\left\vert a_{12}+a_{22}\right\vert
\]
or
\[
f(t_{1})=|a_{21}|+\left\vert a_{12}-a_{22}\right\vert =|a_{11}-a_{21}%
|+\left\vert a_{12}-a_{22}\right\vert .
\]
We thus have again the expression given in (B).

$\circ$ Subcase 2. If $a_{12}\neq a_{22}$ and $a_{12}\neq-a_{22}$ , then
$a_{12}^{2}+a_{22}^{2}+2a_{12}a_{22}\cos t\neq0$ for all $t$ and $f^{\prime
}(t)$ exists for all $t;$ thus we again obtain the expression of (B).

The situation (3) is similar to (2).

If we have (4) and $\left(  a_{11}=\pm a_{21}\text{ or }a_{12}=\pm
a_{22}\right)  $ we proceed as in the second case of (A). If $a_{11}\neq\pm
a_{21}$ and $a_{12}\neq\pm a_{22}$ we are encompassed by Case 1 or Case 2 of (B).
\end{proof}

For real scalars, for the obvious reasons, the expression of the norm is less complicated:

\begin{proposition}
\label{889}Let $T:c_{0}\times c_{0}\rightarrow\mathbb{R}$ be given by
$T(x,y)=\displaystyle\sum_{i,j=1}^{2}a_{ij}x_{i}y_{j},$ with $a_{ij}%
\in\mathbb{R}$. Then%

\[
\Vert T\Vert=\max\{|a_{11}+a_{21}|+|a_{12}+a_{22}|,|a_{11}-a_{21}%
|+|a_{12}-a_{22}|\}.
\]

\end{proposition}

\begin{proof}
As in the proof of the complex case,
\[
\Vert T\Vert=\sup\{\Vert T_{x}\Vert:\Vert x\Vert_{\infty}=1\},
\]
where $T_{x}:\ell_{\infty}^{2}(\mathbb{R})\rightarrow\mathbb{R}$ is given by
\[
T_{x}(y)=(a_{11}x_{1}+a_{21}x_{2})y_{1}+(a_{12}x_{1}+a_{22}x_{2})y_{2}.
\]
We thus have
\[
\Vert T\Vert=\sup\{\Vert T_{x}\Vert:\Vert x\Vert_{\infty}=1\}=\sup
\{|a_{11}x_{1}+a_{21}x_{2}|+|a_{12}x_{1}+a_{22}x_{2}|:\Vert x\Vert_{\infty
}=1\}.
\]
So, we shall maximize
\[
f(x)=|a_{11}x_{1}+a_{21}x_{2}|+|a_{12}x_{1}+a_{22}x_{2}|
\]
with the restriction $\Vert x\Vert_{\infty}=1$. There is no loss of generality
in supposing that $|x_{1}|=1$. If $x_{1}=1$ (the case $x_{1}=-1$ is similar),
then
\[
f(t)=|a_{11}+a_{21}t|+|a_{12}+a_{22}t|.
\]
So, we shall maximize $f$ under the restriction $|t|\leq1$. Thus, invoking
Lemma \ref{888000} the maximum of $f$ is attained at $t=-1$ or when $t=1,$ and
it is given by
\[
\max\{|a_{11}+a_{21}|+|a_{12}+a_{22}|,|a_{11}-a_{21}|+|a_{12}-a_{22}|\}.
\]

\end{proof}

\section{Geometry of the unit ball of $\mathcal{L}(^{2}\ell_{\infty}%
^{2}(\mathbb{R}))$: extreme and exposed points\label{real}}

As mentioned in the Introduction, given a Banach space $E$ and a convex set
$A\subset E$, a vector $x\in A$ is an \textit{extreme point} of $A$ if $y,z\in
A$ with $x=\frac{y+z}{2}$ implies $y=z.$ If $x\in A$ and there is a linear
functional $f\in E^{\ast}$ such that $f(x)=1=\left\Vert f\right\Vert $ and
$f(y)<1$ for all $y\in A\setminus\{x\}$, then $x$ is called \textit{exposed
point}. It is not difficult to prove that exposed points are extreme points.
In this section we obtain all extreme and exposed points of the closed unit
ball of $\mathcal{L}(^{2}\ell_{\infty}^{2}(\mathbb{R})).$

\begin{theorem}
The extreme points of the closed unit ball of $\mathcal{L}(^{2}\ell_{\infty
}^{2}(\mathbb{R}))$ are
\begin{align*}
&  \pm x_{1}y_{1},\pm x_{2}y_{1},\pm x_{1}y_{2},\pm x_{2}y_{2},\\
&  \frac{1}{2}\left(  \pm x_{1}y_{1}\pm x_{2}y_{1}\pm x_{1}y_{2}\mp x_{2}%
y_{2}\right)  ,\\
&  \frac{1}{2}\left(  \mp x_{1}y_{1}\pm x_{2}y_{1}\pm x_{1}y_{2}\pm x_{2}%
y_{2}\right)  ,\\
&  \frac{1}{2}\left(  \pm x_{1}y_{1}\mp x_{2}y_{1}\pm x_{1}y_{2}\pm x_{2}%
y_{2}\right)  ,\\
&  \frac{1}{2}\left(  \pm x_{1}y_{1}\pm x_{2}y_{1}\mp x_{1}y_{2}\pm x_{2}%
y_{2}\right)  .
\end{align*}

\end{theorem}

\begin{proof}
For the sake of simplicity we shall denote $\ell_{\infty}^{2}(\mathbb{R})$ by
$\ell_{\infty}^{2}$ along this proof. Let $T\in B_{\mathcal{L}(^{2}%
\ell_{\infty}^{2})}$ be given by $T(x,y)=ax_{1}y_{1}+bx_{2}y_{1}+cx_{1}%
y_{2}+dx_{2}y_{2}$. By symmetry, it suffices to consider the following cases,
with $a,b,c,d\neq0$:

(1) $T(x,y)=ax_{1}y_{1}$;

(2) $T(x,y)=ax_{1}y_{1}+bx_{2}y_{1}$;

(3) $T(x,y)=ax_{1}y_{1}+bx_{2}y_{1}+cx_{1}y_{2}$;

(4) $T(x,y)=ax_{1}y_{1}+bx_{2}y_{1}+cx_{1}y_{2}+dx_{2}y_{2}$.

Since $T\in B_{\mathcal{L}(^{2}\ell_{\infty}^{2})}$, we know that
$|a|,|b|,|c|$ and $|d|$ are not bigger than $1$.

Case (1). If $|a|<1,$ let $0<\varepsilon<1-|a|$. Defining
\[
A(x,y)=(a+\varepsilon)x_{1}y_{1}%
\]%
\[
B(x,y)=(a-\varepsilon)x_{1}y_{1},
\]
we have $\frac{1}{2}(A+B)=T$ and $A,B\in B_{\mathcal{L}(^{2}\ell_{\infty}%
^{2})}$. Thus, $T$ is not an extreme point. If $|a|=1$, we can suppose $a=1$.
Thus, if there are $A,B\in B_{\mathcal{L}(^{2}\ell_{\infty}^{2})}$ such that
$\frac{1}{2}(A+B)=T$, say,%
\[
A(x,y)=\alpha x_{1}y_{1}+\beta x_{2}y_{1}+\gamma x_{1}y_{2}+\delta x_{2}y_{2}%
\]%
\[
B(x,y)=\alpha^{\prime}x_{1}y_{1}+\beta^{\prime}x_{2}y_{1}+\gamma^{\prime}%
x_{1}y_{2}+\delta^{\prime}x_{2}y_{2},
\]
we have $\left(  \alpha,\beta,\gamma,\delta\right)  =\left(  2-\alpha^{\prime
},-\beta^{\prime},-\gamma^{\prime},-\delta^{\prime}\right)  .$ Since
$|\alpha|,|\alpha^{\prime}|\leq1$, we conclude that $\alpha=\alpha^{\prime}%
=1$. Note that if $\beta\neq0$, then $1+\beta$ or $1-\beta$ is bigger than
$1$. Estimating $A((1,1),(1,0))$ and $A((1,1),(-1,0))$ we conclude that $\Vert
A\Vert>1$ and the same happens to $B$; therefore $\beta=0$. The same argument
shows us that $\gamma=\delta=0$. Thus, $T$ is an extreme point.

Case (2). Note that
\[
\Vert T\Vert=|a|+|b|\leq1.
\]
Let $0<\varepsilon<\min\{|a|,|b|\},$ and defining
\[
A(x,y)=(a+sgn(a)\varepsilon)x_{1}y_{1}+(b-sgn(b)\varepsilon)x_{2}y_{1}%
\]%
\[
B(x,y)=(a-sgn(a)\varepsilon)x_{1}y_{1}+(b+sgn(b)\varepsilon)x_{2}y_{1},
\]
we conclude that $\frac{1}{2}(A+B)=T$ and $A,B\in B_{\mathcal{L}(^{2}%
\ell_{\infty}^{2})}$. Thus, $T$ is not an extreme point.

Case (3). By Proposition \ref{889}, we have
\[
\Vert T\Vert=\max\{|a+b|+|c|,|a-b|+|c|\}.
\]
Note that
\[
|a+b|+|c|=|a-b|+|c|\Leftrightarrow a=0\text{ or }b=0.
\]
Let us consider two subcases:

(3A) $ab>0$;

(3B) $ab<0$.

If (3A) happens, then $|a-b|+|c|<1$. Defining $0<\varepsilon<\frac
{1-(|a-b|+|c|)}{2}$ and
\[
A(x,y)=(a+\varepsilon)x_{1}y_{1}+(b-\varepsilon)x_{2}y_{1}+cx_{1}y_{2}%
\]%
\[
B(x,y)=(a-\varepsilon)x_{1}y_{1}+(b+\varepsilon)x_{2}y_{1}+cx_{1}y_{2},
\]
we have $\frac{1}{2}(A+B)=T$ and $A,B\in B_{\mathcal{L}(^{2}\ell_{\infty}%
^{2})}$. Thus $T$ is not extreme point.

If (3B) happens, then $|a+b|+|c|<1$. Defining $0<\varepsilon<\frac
{1-(|a+b|+|c|)}{2}$ and
\[
A(x,y)=(a+\varepsilon)x_{1}y_{1}+(b+\varepsilon)x_{2}y_{1}+cx_{1}y_{2}%
\]%
\[
B(x,y)=(a-\varepsilon)x_{1}y_{1}+(b-\varepsilon)x_{2}y_{1}+cx_{1}y_{2},
\]
we have $\frac{1}{2}(A+B)=T$ and $A,B\in B_{\mathcal{L}(^{2}\ell_{\infty}%
^{2})}$. Thus $T$ is not extreme point.

Case (4). We consider four subcases:

(4A) $ab>0\text{ and }cd>0$;

(4B) $ab<0\text{ and }cd<0$;

(4C) $ab>0\text{ and }cd<0$;

(4D) $ab<0\text{ and }cd>0$.

If (4A) happens, then $|a-b|+|c-d|<1$. Considering $0<\varepsilon
<\frac{1-(|a-b|+|c-d|)}{2}$ and defining
\[
A(x,y)=(a+\varepsilon)x_{1}y_{1}+(b-\varepsilon)x_{2}y_{1}+cx_{1}y_{2}%
+dx_{2}y_{2}%
\]%
\[
B(x,y)=(a-\varepsilon)x_{1}y_{1}+(b+\varepsilon)x_{2}y_{1}+cx_{1}y_{2}%
+dx_{2}y_{2},
\]
we have $\frac{1}{2}(A+B)=T$ and $A,B\in B_{\mathcal{L}(^{2}\ell_{\infty}%
^{2})}$. Thus $T$ is not an extreme point.

If (4B) happens, then $|a+b|+|c+d|<1$. Considering $0<\varepsilon
<\frac{1-(|a+b|+|c+d|)}{2}$ and defining
\[
A(x,y)=(a+\varepsilon)x_{1}y_{1}+(b+\varepsilon)x_{2}y_{1}+cx_{1}y_{2}%
+dx_{2}y_{2}%
\]%
\[
B(x,y)=(a-\varepsilon)x_{1}y_{1}+(b-\varepsilon)x_{2}y_{1}+cx_{1}y_{2}%
+dx_{2}y_{2},
\]
we have $\frac{1}{2}(A+B)=T$ and $A,B\in B_{\mathcal{L}(^{2}\ell_{\infty}%
^{2})}$. Thus $T$ is not an extreme point.

If (4C) happens we can assume $a,b,c>0$ and $d<0$. Note that by Proposition
\ref{889},
\begin{equation}
\Vert T\Vert=|a+b|+|c+d| \label{9pp}%
\end{equation}
or
\begin{equation}
\Vert T\Vert=|a-b|+|c-d|. \label{9ii}%
\end{equation}
We shall consider just (\ref{9pp}) because (\ref{9ii}) is similar. If we have
(\ref{9pp}) then, by Lemma \ref{maxpos}, there are two possibilities:

(4CA) $a\geq-d,\ b\geq-d$ and $c\geq-d;$

(4CB) $a\geq c,\ b\geq c$ and $-d\geq c.$

We shall first prove that if
\[
card\{a,b,c,-d\}\neq1,
\]
then $T$ is not an extreme point. Let us first suppose (4CA).

If $card\{a,b,c,-d\}\neq1$ we can assume $a\neq b$ because the other cases are
analogous. We thus have two possibilities:

(4CAA)\ $a>b,$

(4CAB) $a<b.$

Let us first consider (4CAA):

Since $a>b$, we have $a>b\geq-d$ and $c\geq-d$. We consider two cases:

(4CAAA) $a>b>-d$ e $c\geq-d$;

(4CAAB) $a>b=-d$ e $c\geq-d$.

If (4CAAA) happens, since $a>b>-d>0$ and $c\geq-d>0$ we conclude that
\[
a+b+c+d>a-b+c-d,
\]
i.e.,%
\[
\left\vert a+b\right\vert +\left\vert c+d\right\vert >\left\vert
a-b\right\vert +\left\vert c-d\right\vert
\]
and thus, by Proposition \ref{889}, $|a-b|+|c-d|<1$. Considering $0<\varepsilon
<\frac{1-(|a-b|+|c-d|)}{2}$ e defining
\[
A(x,y)=(a+\varepsilon)x_{1}y_{1}+(b-\varepsilon)x_{2}y_{1}+cx_{1}y_{2}%
+dx_{2}y_{2}%
\]%
\[
B(x,y)=(a-\varepsilon)x_{1}y_{1}+(b+\varepsilon)x_{2}y_{1}+cx_{1}y_{2}%
+dx_{2}y_{2},
\]
we have $\frac{1}{2}(A+B)=T$ and $A,B\in B_{\mathcal{L}(^{2}\ell_{\infty}%
^{2})}$. Thus $T$ is not an extreme point.

If (4CAAB) happens, since $b=-d$ and using that $a>b>0$ and $c\geq-d>0$ we
have%
\[
\left\vert a+b\right\vert +\left\vert c+d\right\vert =\left\vert
a-b\right\vert +\left\vert c-d\right\vert =a+c
\]
and, by then by Proposition \ref{889},
\[
\Vert T\Vert=a+c\leq1.
\]

We have two possibilities:

(4CAABA) $a+c<1;$

(4CAABB) $a+c=1.$

If (4CAABA) happens, we choose $0<\varepsilon<\min\{a-b,1-(a+c)\}$ and define
\[
A(x,y)=(a+\varepsilon)x_{1}y_{1}+bx_{2}y_{1}+cx_{1}y_{2}-bx_{2}y_{2}%
\]%
\[
B(x,y)=(a-\varepsilon)x_{1}y_{1}+bx_{2}y_{1}+cx_{1}y_{2}-bx_{2}y_{2},
\]
and by Lemma \ref{maxig}, we conclude that $A,B\in B_{\mathcal{L}(^{2}%
\ell_{\infty}^{2})}$ and $\frac{1}{2}(A+B)=T$. Hence, $T$ is not an extreme point.

If (4CAABB) happens, we can write
\[
T(x,y)=ax_{1}y_{1}+bx_{2}y_{1}+(1-a)x_{1}y_{2}-bx_{2}y_{2}.
\]
Since $c\geq-d$, it follows that
\[
1-a\geq b.
\]
If $1-a=b$, then
\[
T(x,y)=(1-b)x_{1}y_{1}+bx_{2}y_{1}+bx_{1}y_{2}-bx_{2}y_{2}.
\]
Note that
\[
|\left(  1-b\right)  -b|+|b-(-b)|=|1-2b|+2b.
\]
Since $1-b=a>b>0$, it follows that $0<b<\frac{1}{2}$. Considering
$0<\varepsilon<\min\{b,\frac{a-b}{2}\}$ and defining
\[
A(x,y)=(1-b+\varepsilon)x_{1}y_{1}+(b-\varepsilon)x_{2}y_{1}+(b-\varepsilon
)x_{1}y_{2}+(-b+\varepsilon)x_{2}y_{2}%
\]%
\[
B(x,y)=(1-b-\varepsilon)x_{1}y_{1}+(b+\varepsilon)x_{2}y_{1}+(b+\varepsilon
)x_{1}y_{2}+(-b-\varepsilon)x_{2}y_{2},
\]
we have $\frac{1}{2}(A+B)=T$ and $\Vert A\Vert=\Vert B\Vert=1$. Hence, $T$ is
not an extreme point.

If $1-a>b$, then $b<a<1-b$. Considering $0<\varepsilon<\min\{a-b,1-a-b\}$ and
defining
\[
A(x,y)=(a+\varepsilon)x_{1}y_{1}+bx_{2}y_{1}+(1-a-\varepsilon)x_{1}%
y_{2}-bx_{2}y_{2}%
\]%
\[
B(x,y)=(a-\varepsilon)x_{1}y_{1}+bx_{2}y_{1}+(1-a+\varepsilon)x_{1}%
y_{2}-bx_{2}y_{2},
\]
we conclude that $A,B\in B_{\mathcal{L}(^{2}\ell_{\infty}^{2})}$ and $\frac
{1}{2}(A+B)=T$. Thus $T$ is not a extreme point.

Now let us prove (4CAB). Since $b>a$, then $b>a\geq-d\text{ and }c\geq-d.$

If $b>a>-d$ and $c\geq-d$, then
\[
a>-d\Rightarrow a+d>-a-d\Rightarrow a+b+c+d>b-a+c-d.
\]
Hence
\[
|a+b|+|c+d|>|a-b|+|c-d|.
\]
Considering $0<\varepsilon<\frac{1-(|a-b|+|c-d|)}{2}$ and defining
\[
A(x,y)=(a+\varepsilon)x_{1}y_{1}+(b-\varepsilon)x_{2}y_{1}+cx_{1}y_{2}%
+dx_{2}y_{2}%
\]%
\[
B(x,y)=(a-\varepsilon)x_{1}y_{1}+(b+\varepsilon)x_{2}y_{1}+cx_{1}y_{2}%
+dx_{2}y_{2},
\]
we conclude that $\frac{1}{2}(A+B)=T$ and $A,B\in B_{\mathcal{L}(^{2}%
\ell_{\infty}^{2})}$. Thus $T$ is not an extreme point.

If $b>a=-d$ and $c\geq-d$, then we shall proceed as in the case (4CAAB) to
observe that
\[
T(x,y)=ax_{1}y_{1}+bx_{2}y_{1}+cx_{1}y_{2}-ax_{2}y_{2}%
\]
is not an extreme point.

So, it remains to look for extreme points in the case (4C) when
\[
card\{a,b,c,-d\}=1.
\]
In this case we can write
\begin{equation}
T(x,y)=ax_{1}y_{1}+ax_{2}y_{1}+ax_{1}y_{2}-ax_{2}y_{2}. \label{nbv}%
\end{equation}
Since $2a=\Vert T\Vert\leq1$, we have $a\leq\frac{1}{2}.$ If $a<\frac{1}{2}$,
$T$ is not an extreme point. Let us show that when $a=\frac{1}{2}$ the
bilinear form $T$ given by (\ref{nbv}) is an extreme point.

Suppose that there exist $A,B\in B_{\mathcal{L}(^{2}\ell_{\infty}^{2})}$ such
that $\frac{1}{2}(A+B)=T$. Denoting
\[
A(x,y)=\alpha x_{1}y_{1}+\beta x_{2}y_{1}+\gamma x_{1}y_{2}+\delta x_{2}%
y_{2},
\]%
\[
B(x,y)=\alpha^{\prime}x_{1}y_{1}+\beta^{\prime}x_{2}y_{1}+\gamma^{\prime}%
x_{1}y_{2}+\delta^{\prime}x_{2}y_{2},
\]
we have
\[
\left(  \alpha+\alpha^{\prime},\beta+\beta^{\prime},\gamma+\gamma^{\prime
},\delta+\delta^{\prime}\right)  =\left(  1,1,1,-1\right)  .
\]
Since $|\alpha|,|\alpha^{\prime}|\leq1$, it follows that $\alpha\in
\lbrack0,1]$. A similar argument tells us that $\beta,\gamma,-\delta\in
\lbrack0,1]$. We claim that if $\alpha\neq\frac{1}{2}$, then $A\notin
B_{\mathcal{L}(^{2}\ell_{\infty}^{2}(\mathbb{R}))}$ or $B\notin B_{\mathcal{L}%
(^{2}\ell_{\infty}^{2}(\mathbb{R}))}$. Note that
\begin{align*}
0\leq\alpha<\frac{1}{2}  &  \Rightarrow-\frac{1}{2}<-\alpha\leq0\\
&  \Rightarrow\frac{1}{2}<1-\alpha\leq1\\
&  \Rightarrow\frac{1}{2}<\alpha^{\prime}\leq1
\end{align*}
and
\[
\frac{1}{2}<\alpha\leq1\Rightarrow0\leq\alpha^{\prime}<\frac{1}{2}.
\]
In a similar fashion,%
\[
0\leq\beta<\frac{1}{2}\Rightarrow\frac{1}{2}<\beta^{\prime}\leq1
\]
and so on.

So, let us first suppose $\alpha\in\lbrack{0,\frac{1}{2}})$. We may have
$\beta\in\lbrack0,\frac{1}{2}]$ or $\beta\in(\frac{1}{2},1]$.

If $\beta\in\lbrack0,\frac{1}{2}]$, then $\alpha^{\prime}\in(\frac{1}{2},1]$
and $\beta^{\prime}\in\lbrack\frac{1}{2},1]$. Therefore,
\[
B((1,1),(1,0))=\alpha^{\prime}+\beta^{\prime}>1
\]
and thus $\Vert B\Vert>1$, a contradiction.

If $\beta\in(\frac{1}{2},1]$, then we may have:

(P1) $\gamma\in\lbrack0,\frac{1}{2}]$ and $\delta\in\lbrack-1,-\frac{1}{2}]$;

(P2) $\gamma\in\lbrack0,\frac{1}{2}]$ and $\delta\in(-\frac{1}{2},0]$;

(P3) $\gamma\in(\frac{1}{2},1]$ and $\delta\in\lbrack-1,-\frac{1}{2}]$;

(P4) $\gamma\in(\frac{1}{2},1]$ and $\delta\in(-\frac{1}{2},0]$.

If (P1) holds, then $\alpha^{\prime}\in(\frac{1}{2},1]$, $\gamma^{\prime}%
\in\lbrack\frac{1}{2},1]$ and $\beta^{\prime},-\delta^{\prime}\in
\lbrack0,\frac{1}{2})$. Thus $\alpha^{\prime}>\beta^{\prime}$, $\gamma
^{\prime}\geq\beta^{\prime}$ and $\alpha^{\prime},\gamma^{\prime}\geq
-\delta^{\prime}$. When $\beta^{\prime}\geq-\delta^{\prime}$, by Lemma
\ref{maxpos}, we have
\[
\Vert B\Vert=\alpha^{\prime}+\beta^{\prime}+\gamma^{\prime}+\delta^{\prime
}>1.
\]
When $-\delta^{\prime}\geq\beta^{\prime}$, by Lemma \ref{maxneg}, we have%
\[
\Vert B\Vert=\alpha^{\prime}-\beta^{\prime}+\gamma^{\prime}-\delta^{\prime
}>1.
\]

If (P2) holds, then $\alpha^{\prime},-\delta^{\prime}\in(\frac{1}{2},1]$,
$\gamma^{\prime}\in\lbrack\frac{1}{2},1]$ and $\beta^{\prime}\in\lbrack
0,\frac{1}{2})$. Thus $\alpha^{\prime},-\delta^{\prime}>\beta^{\prime}$ and
$\gamma^{\prime}\geq\beta^{\prime}$. By Lemma \ref{maxneg}, we have%
\[
\Vert B\Vert=\alpha^{\prime}-\beta^{\prime}+\gamma^{\prime}-\delta^{\prime
}>1.
\]

If (P3) holds, then
\[
A((1,-1),(0,1))=-\delta+\gamma>1
\]
and thus $\Vert A\Vert>1$.

If we have (P4), then $\beta,\gamma\geq\alpha$ and $\beta,\gamma\geq-\delta$.
When $\alpha\geq-\delta$, by Lemma \ref{maxpos}, we have%
\[
\Vert A\Vert=\alpha+\beta+\gamma+\delta>1.
\]
When $-\delta\geq\alpha$, by Lemma \ref{maxneg}, we have
\[
\Vert A\Vert=\beta-\alpha+\gamma-\delta>1.
\]

Now, let us suppose $\alpha\in(\frac{1}{2},1]$. We may have $\beta\in
\lbrack0,\frac{1}{2}]$ or $\beta\in\lbrack\frac{1}{2},1]$. If $\beta\in
\lbrack\frac{1}{2},1]$, then%
\[
A((1,1),(1,0))=\alpha+\beta>1
\]
and hence$\Vert A\Vert>1$, a contradiction. If $\beta\in\lbrack0,\frac{1}{2}%
]$, we may have:

(K1) $\gamma\in\lbrack0,\frac{1}{2}]$ and $\delta\in\lbrack-1,-\frac{1}{2}]$;

(K2) $\gamma\in\lbrack0,\frac{1}{2}]$ and $\delta\in(-\frac{1}{2},0]$;

(K3) $\gamma\in(\frac{1}{2},1]$ and $\delta\in\lbrack-1,-\frac{1}{2}]$;

(K4) $\gamma\in(\frac{1}{2},1]$ and $\delta\in(-\frac{1}{2},0]$.

If (K1) happens, then $\alpha>\beta$, $-\delta\geq\beta$ and $\alpha
,-\delta\geq\gamma$. When $\beta\geq\gamma$, by Lemma \ref{maxpos}, we have
\[
\Vert A\Vert=\alpha+\beta-\delta-\gamma>1.
\]
When $\gamma\geq\beta$, by Lemma \ref{maxneg}, we have
\[
\Vert A\Vert=\alpha-\beta+\gamma-\delta>1.
\]

If (K2) occurs, then $\gamma^{\prime}\in\lbrack\frac{1}{2},1]$ e
$-\delta^{\prime}\in(\frac{1}{2},1]$. Therefore,
\[
B((1,-1),(0,1))=-\delta^{\prime}+\gamma^{\prime}>1
\]
and $\Vert B\Vert>1$.

If we have (K3), then
\[
A((1,-1),(0,1))=-\delta+\gamma>1,
\]
and so $\Vert A\Vert>1.$

If (K4) happens, then $\alpha,\gamma\geq\beta$ and $\alpha,\gamma\geq-\delta$.
When $\beta\geq-\delta$, by Lemma \ref{maxpos}, we have
\[
\Vert A\Vert=\alpha+\beta+\gamma+\delta>1.
\]
When $-\delta\geq\beta$, by Lemma \ref{maxneg}, we have
\[
\Vert A\Vert=\alpha-\beta+\gamma-\delta>1.
\]
The case (4CB) is analogous to (4CA) and (4D) is similar to (4C).
\end{proof}

\begin{theorem}
The extreme and exposed points of $B_{\mathcal{L}(^{2}\ell_{\infty}%
^{2}(\mathbb{R}))}$ are the same.
\end{theorem}

\begin{proof}
It suffices to prove that $extB_{\mathcal{L}(^{2}\ell_{\infty}^{2}%
(\mathbb{R}))}\subseteq expB_{\mathcal{L}(^{2}\ell_{\infty}^{2}(\mathbb{R}))}$.

Let us prove that $x_{1}y_{1}$ is an exposed point. Define a linear form
$f:\mathcal{L}(^{2}\ell_{\infty}^{2}(\mathbb{R}))\rightarrow\mathbb{R}$ by
$f(x_{1}y_{1})=1$ and $f(x_{r}y_{s})=0$, for $\left(  r,s\right)  \neq\left(
1,1\right)  $. Thus $\Vert f\Vert=1=f\left(  x_{1}y_{1}\right)  $ and
\[
f(T)<1
\]
for all $T=ax_{1}y_{1}+bx_{2}y_{1}+cx_{1}y_{2}+dx_{2}y_{2}$ in $B_{\mathcal{L}%
(^{2}\ell_{\infty}^{2}(\mathbb{R}))}\backslash\{x_{1}y_{1}\}.$ In fact, note
that $a\leq1,$ otherwise $\left\Vert T\right\Vert >1$. If $a=1$, then $b=0$,
because
\[
T((1,1),(1,0))=a+b\text{ and }T((1,-1),(1,0))=a-b.
\]
The same argument shows that $c=d=0.$ Thus $a<1$ and $f(T)=a<1.$ A similar
argument shows that $\pm x_{1}y_{1},\pm x_{2}y_{1},\pm x_{1}y_{2},\pm
x_{2}y_{2}$ are exposed points.

Now, let us prove that $T(x,y)=\frac{1}{2}(x_{1}y_{1}+x_{2}y_{1}+x_{1}%
y_{2}-x_{2}y_{2})$ is an exposed point. Define a linear form such that
$f(x_{1}y_{1})=\frac{1}{2}$, $f(x_{2}y_{1})=\frac{1}{2}$, $f(x_{1}y_{2}%
)=\frac{1}{2}$ and $f(x_{2}y_{2})=-\frac{1}{2}$. Thus $f(T)=1$. Note that
$\Vert f\Vert=1$. In fact, by Proposition \ref{889}, we have%

\[
\Vert L\Vert=\max\left\{  |a+b|+|c+d|,|a-b|+|c-d|\right\}  ,
\]
for all $L(x,y)=ax_{1}y_{1}+bx_{2}y_{1}+cx_{1}y_{2}+dx_{2}y_{2}$. Therefore,
if $L\in B_{\mathcal{L}(^{2}\ell_{\infty}^{2}(\mathbb{R}))}$, then
\begin{equation}
|a+b|\leq1\text{ and }|c-d|\leq1 \label{nhy}%
\end{equation}
and%
\begin{align*}
\Vert f\Vert &  =\sup\left\{  |f(L)|:L\in B_{\mathcal{L}(^{2}\ell_{\infty}%
^{2}(\mathbb{R}))}\right\} \\
&  =\sup\left\{  \left\vert \frac{a+b+c-d}{2}\right\vert :ax_{1}y_{1}%
+bx_{2}y_{1}+cx_{1}y_{2}+dx_{2}y_{2}\in B_{\mathcal{L}(^{2}\ell_{\infty}%
^{2}(\mathbb{R}))}\right\} \\
&  \leq\frac{1}{2}\sup\left\{  |a+b|+|c-d|:ax_{1}y_{1}+bx_{2}y_{1}+cx_{1}%
y_{2}+dx_{2}y_{2}\in B_{\mathcal{L}(^{2}\ell_{\infty}^{2}(\mathbb{R}%
))}\right\} \\
&  \leq\frac{1}{2}(1+1)=1.
\end{align*}
Moreover, if $f(ax_{1}y_{1}+bx_{2}y_{1}+cx_{1}y_{2}+dx_{2}y_{2})=1$ for a
certain bilinear form $ax_{1}y_{1}+bx_{2}y_{1}+cx_{1}y_{2}+dx_{2}y_{2}\in
B_{\mathcal{L}(^{2}\ell_{\infty}^{2}(\mathbb{R}))}$, then
\[
\frac{a+b+c-d}{2}=1
\]
and thus%
\begin{equation}
\frac{\left\vert a+b\right\vert +\left\vert c-d\right\vert }{2}\geq1.
\label{711}%
\end{equation}
By (\ref{nhy}) and (\ref{711}) we conclude that $|a+b|=1$ and $|c-d|=1$. Thus,
by Proposition \ref{889}, we have $|a-b|=0$ and $|c+d|=0$ and since
$\frac{a+b+c-d}{2}=1$ we conclude that $a=b=c=-d=\frac{1}{2}$

A similar argument shows that the other extreme points are also exposed points.
\end{proof}

\section{Littlewood's $4/3$ inequality and an optimization problem: real case}

Extreme points are important for optimization of convex continuous functions
for a very simple reason: first we shall recall a theorem due to
Minkowski/Krein-Milman which asserts that if $E$ is a locally convex space and
$K$ is a nonempty convex and compact subset of $E,$ then $K$ has at least an
extreme point and $K=\overline{conv}(extK)$, where $extK$ is the set of all
extreme points of $K$. If $f:K\rightarrow\mathbb{R}$ is a convex continuous
function its maximum is attained in an extreme point $k_{0}\in K$. In fact,
suppose that $k_{0}\in K$ is a point where the maximum is attained; the
Minkowski/Krein-Milman asserts that there are $\lambda_{1},...,\lambda_{n}%
\in\lbrack0,1]$ such that
\[
k_{0}=%
%TCIMACRO{\tsum \limits_{j=1}^{n}}%
%BeginExpansion
{\textstyle\sum\limits_{j=1}^{n}}
%EndExpansion
\lambda_{j}k_{j},
\]
with $k_{1},...,k_{n}\in extK$ and $%
%TCIMACRO{\tsum \limits_{j=1}^{n}}%
%BeginExpansion
{\textstyle\sum\limits_{j=1}^{n}}
%EndExpansion
\lambda_{j}=1.$ If the maximum of $f$ is not attained in any extreme point,
then%
\[
f\left(  k_{0}\right)  \leq%
%TCIMACRO{\tsum \limits_{j=1}^{n}}%
%BeginExpansion
{\textstyle\sum\limits_{j=1}^{n}}
%EndExpansion
\lambda_{j}f\left(  k_{j}\right)  <%
%TCIMACRO{\tsum \limits_{j=1}^{n}}%
%BeginExpansion
{\textstyle\sum\limits_{j=1}^{n}}
%EndExpansion
\lambda_{j}f\left(  k_{0}\right)  =f\left(  k_{0}\right)  ,
\]
a contradiction. However, it is plain that the maximum may also be attained in
non-extreme points. For instance, $f:B_{\ell_{\infty}(\mathbb{R})}%
\rightarrow\mathbb{R}$ given by $f(x)=\left\Vert x\right\Vert $ attains its
maximum in all points of the unit sphere of $\ell_{\infty}(\mathbb{R})$ but
--for instance-- the canonical vectors $e_{j}$ are not extreme points.

For $\mathbb{K}=\mathbb{R}$ or $\mathbb{C}$, Littlewood's $4/3$ inequality
asserts that there exists a sequence of positive scalars $B_{2}^{\mathbb{K}}$
in $[1,\infty)$ such that%
\[
\left(  \sum\limits_{i,j=1}^{\infty}\left\vert U(e_{i},e_{j})\right\vert
^{\frac{4}{3}}\right)  ^{\frac{3}{4}}\leq B_{2}^{\mathbb{K}}\left\Vert
U\right\Vert
\]
for all continuous $2$-linear forms $U:c_{0}\times c_{0}\rightarrow\mathbb{K}%
$. For $\mathbb{K}=\mathbb{R}$ the optimal constant is $\sqrt{2}$ \cite{diniz}
and for complex scalars all that is known is that
\[
1\leq B_{2}^{\mathbb{C}}\leq\frac{2}{\sqrt{\pi}}.
\]
Littlewood's $4/3$ inequality is a forerunner of the classical
Bohnenblust--Hille inequality. The Bohnenblust--Hille inequality for
$m$-linear forms (\cite{bh}) tells us that there exists a sequence of positive
scalars $\left(  B_{m}^{\mathbb{K}}\right)  _{m=1}^{\infty}$ in $[1,\infty)$
such that%
\[
\left(  \sum\limits_{i_{1},\ldots,i_{m}=1}^{\infty}\left\vert U(e_{i_{^{1}}%
},\ldots,e_{i_{m}})\right\vert ^{\frac{2m}{m+1}}\right)  ^{\frac{m+1}{2m}}\leq
B_{m}^{\mathbb{K}}\left\Vert U\right\Vert
\]
for all continuous $m$-linear forms $U:c_{0}\times\cdots\times c_{0}%
\rightarrow\mathbb{K}$. The investigation of the optimal constants in the
Bohnenblust--Hille inequality can be written as the following optimization
problem:
\[
B_{\mathbb{K},m}:=\inf\left\{  \left(  \sum_{j_{1},\cdots,j_{m}=1}^{\infty
}\left\vert T(e_{j_{1}},\cdots,e_{j_{m}})\right\vert ^{\frac{2m}{m+1}}\right)
^{\frac{m+1}{2m}}\!\!\!\!\!,\text{ among all }m\text{--linear forms }T,\text{
with }\Vert T\Vert=1\right\}  .
\]
This optimization problem is a rather challenging, still surrounded by many
mysteries. For recent developments related to the search of optimal constants
we refer to \cite{bohr, pt} and references therein.

We shall call \textit{optimal bilinear} form any $T$ satisfying the
optimization problem above.

By \cite{diniz} we know that the constant $\sqrt{2}$ is sharp for real scalars
when $m=2$, and it is attained when considering the bilinear form
\begin{equation}
T(x,y)=x_{1}y_{1}+x_{1}y_{2}+x_{2}y_{1}-x_{2}y_{2}. \label{pmmm}%
\end{equation}
The next theorem shows that when considering bilinear forms $T:c_{0}\times
c_{0}\rightarrow\mathbb{R}$ of the form $T(x,y)=\displaystyle\sum_{i,j=1}%
^{2}a_{ij}x_{i}y_{j}$ all extremal bilinear forms are very close to
(\ref{pmmm}). It also shows that there are no (norm one) optimal bilinear
forms outside the set of extreme points of the closed unit ball of
$\mathcal{L}(^{2}\ell_{\infty}^{2}(\mathbb{R})).$

\begin{theorem}
Let $T:\ell_{\infty}^{2}(\mathbb{R})\times\ell_{\infty}^{2}(\mathbb{R}%
)\rightarrow\mathbb{R}$ be given by $T(x,y)=\displaystyle\sum_{i,j=1}%
^{2}a_{ij}x_{i}y_{j},$ with $a_{ij}\in\mathbb{R}$. Then the bilinear forms
satisfying
\[
\left(  \sum\limits_{j,k=1}^{2}\left\vert T(e_{j},e_{k})\right\vert ^{\frac
{4}{3}}\right)  ^{\frac{3}{4}}=\sqrt{2}\left\Vert T\right\Vert
\]
are given by
\[
T(x,y)=\alpha x_{1}y_{1}+\alpha x_{1}y_{2}+\alpha x_{2}y_{1}-\alpha x_{2}y_{2}%
\]
or
\[
T(x,y)=\alpha x_{1}y_{1}+\alpha x_{1}y_{2}-\alpha x_{2}y_{1}+\alpha x_{2}y_{2}%
\]
or
\[
T(x,y)=\alpha x_{1}y_{1}-\alpha x_{1}y_{2}+\alpha x_{2}y_{1}+\alpha x_{2}y_{2}%
\]
or
\[
T(x,y)=-\alpha x_{1}y_{1}+\alpha x_{1}y_{2}+\alpha x_{2}y_{1}+\alpha
x_{2}y_{2}%
\]
for all $\alpha\in\mathbb{R}\backslash\{0\}$.
\end{theorem}

\begin{proof}
If $a_{ij}=0$ for some $i,j\in\{1,2\}$ it is not difficult to prove that the
constant $\sqrt{2}$ is not achieved. Let $T:\ell_{\infty}^{2}(\mathbb{R}%
)\times\ell_{\infty}^{2}(\mathbb{R})\rightarrow\mathbb{R}$ be given by
$T(x,y)=\displaystyle\sum_{i,j=1}^{2}a_{ij}x_{i}y_{j},$ with $a_{ij}%
\in\mathbb{R}\setminus\{0\}$. Let us first suppose that
\[
a_{11}\cdot a_{21}>0\text{ and }a_{12}\cdot a_{22}>0.
\]
In this case, by Proposition \ref{889} we have
\[
\Vert T\Vert=|a_{11}+a_{21}|+|a_{12}+a_{22}|=\Vert(a_{11},a_{21},a_{12}%
,a_{22})\Vert_{1}.
\]
Therefore,
\[
\frac{\Vert(a_{11},a_{21},a_{12},a_{22})\Vert_{\frac{4}{3}}}{\Vert T\Vert
}=\frac{\Vert(a_{11},a_{21},a_{12},a_{22})\Vert_{\frac{4}{3}}}{\Vert
(a_{11},a_{21},a_{12},a_{22})\Vert_{1}}\leq1.
\]
Now, suppose
\[
a_{11},a_{12}>0\text{ and }a_{21},a_{22}<0
\]
Again, using Proposition \ref{889}, we conclude that
\[
\Vert T\Vert=|a_{11}-a_{21}|+|a_{12}-a_{22}|=\Vert(a_{11},a_{21},a_{12}%
,a_{22})\Vert_{1}.
\]
Thus
\[
\frac{\Vert(a_{11},a_{21},a_{12},a_{22})\Vert_{\frac{4}{3}}}{\Vert T\Vert
}=\frac{\Vert(a_{11},a_{21},a_{12},a_{22})\Vert_{\frac{4}{3}}}{\Vert
(a_{11},a_{21},a_{12},a_{22})\Vert_{1}}\leq1.
\]
The cases
\[
a_{11},a_{12}<0\text{ and }a_{21},a_{22}>0,
\]%
\[
a_{11},a_{22}<0\text{ and }a_{12},a_{21}>0,
\]
and
\[
a_{11},a_{22}>0\text{ and }a_{12},a_{21}<0
\]
are similar.

By symmetry, the remaining cases can be summarized in the case
\[
a_{11},a_{12},a_{21}>0\text{ and }a_{22}<0.
\]
By Proposition \ref{889} we know that
\[
\Vert T\Vert=\max\{a_{11}+a_{21}+|a_{12}+a_{22}|,|a_{11}-a_{21}|+a_{12}%
-a_{22}\}
\]
and thus
\begin{equation}
\Vert T\Vert=a_{11}+a_{21}+|a_{12}+a_{22}| \label{76}%
\end{equation}
or
\begin{equation}
\Vert T\Vert=|a_{11}-a_{21}|+a_{12}-a_{22}. \label{00}%
\end{equation}
If (\ref{76}) occurs, we may have
\[
(\alpha)\ \Vert T\Vert=a_{11}+a_{21}+a_{12}+a_{22}%
\]
or
\[
(\beta)\ \Vert T\Vert=a_{11}+a_{21}-a_{12}-a_{22}.
\]
If we have ($\alpha$), note that
\[
a_{12}\geq-a_{22}%
\]
and define
\[
f(x,y,z,w)=\frac{\Vert(x,y,z,w)\Vert_{\frac{4}{3}}}{x+y+z+w}.
\]
Since
\[
f_{x}(x,y,z,w)=\frac{x^{\frac{1}{3}}(x+y+z+w)-\Vert(x,y,z,w)\Vert_{\frac{4}%
{3}}^{\frac{4}{3}}}{\Vert(x,y,z,w)\Vert_{\frac{4}{3}}^{\frac{1}{3}%
}(x+y+z+w)^{2}},
\]%
\[
f_{y}(x,y,z,w)=\frac{y^{\frac{1}{3}}(x+y+z+w)-\Vert(x,y,z,w)\Vert_{\frac{4}%
{3}}^{\frac{4}{3}}}{\Vert(x,y,z,w)\Vert_{\frac{4}{3}}^{\frac{1}{3}%
}(x+y+z+w)^{2}},
\]%
\[
f_{z}(x,y,z,w)=\frac{z^{\frac{1}{3}}(x+y+z+w)-\Vert(x,y,z,w)\Vert_{\frac{4}%
{3}}^{\frac{4}{3}}}{\Vert(x,y,z,w)\Vert_{\frac{4}{3}}^{\frac{1}{3}%
}(x+y+z+w)^{2}},
\]%
\[
f_{w}(x,y,z,w)=\frac{w^{\frac{1}{3}}(x+y+z+w)-\Vert(x,y,z,w)\Vert_{\frac{4}%
{3}}^{\frac{4}{3}}}{\Vert(x,y,z,w)\Vert_{\frac{4}{3}}^{\frac{1}{3}%
}(x+y+z+w)^{2}},
\]
we have
\[
\nabla f=0\Rightarrow x^{\frac{1}{3}}=y^{\frac{1}{3}}=z^{\frac{1}{3}}%
=w^{\frac{1}{3}}\Rightarrow x=y=z=w.
\]
Since $x,y,z>0$ and $w<0$, we conclude that $\nabla f\neq0$.

From now on let $\left(  a,b,c,d)=(a_{11},a_{21},a_{12},a_{22}\right)  .$ By
Lemma \ref{maxpos}
\begin{align*}
&  \left\{  (a,b,c,d):a,b,c,-d>0:\text{ }a+b+c+d\geq|a-b|+c-d\right\}  =\\
\{(a,b,c,d)  &  :a,b,c,-d>0:c\geq-d,b\geq-d,a\geq-d\}\\
\cup\{(a,b,c,d)  &  :a,b,c,-d>0:\text{ }a\geq c,\ b\geq c,\ -d\geq c\}.
\end{align*}
Since in our case $c\geq-d$ we conclude that we shall search the maximum of
$f$ in the set%

\[
H:=\{(a,b,c,d):a,b,c>0,\ d<0\text{ and}\ c\geq-d,b\geq-d,a\geq-d\}.
\]
Defining
\[
A:=\{(a,b,c,d):a,b,c>0,\ d<0\text{ and }a>-d,b>-d,c>-d\},
\]
the maximum of $f$ belongs to $H\backslash A$. Note that
\begin{align*}
H\backslash A  &  =\{(a,b,c,-a):a,b,c>0\text{ and }b\geq a,c\geq a\}\\
&  \cup\{(a,b,c,-b):a,b,c>0\text{ and }a\geq b,c\geq b\}\\
&  \cup\{(a,b,c,-c):a,b,c>0\text{ and }a\geq c,b\geq c\}.
\end{align*}
Let us first consider the set
\[
\{(a,b,c,-a):a,b,c>0\text{ and }b\geq a,c\geq a\}.
\]
In this case, let
\[
g(x,y,z):=f(x,y,z,-x)=\frac{(2x^{\frac{4}{3}}+y^{\frac{4}{3}}+z^{\frac{4}{3}%
})^{\frac{3}{4}}}{y+z}.
\]
Thus
\[
g_{x}=\frac{2x^{\frac{1}{3}}}{\Vert(x,y,z,-x)\Vert_{\frac{4}{3}}^{\frac{1}{3}%
}(y+z)}.
\]
and
\[
g_{x}=0\Leftrightarrow x=0.
\]
Hence, the maximum of $g$ does not belong to $\{(a,b,c,-a):a,b,c>0\text{ and
}b>a,c>a\}$, i.e., the maximum belongs to
\[
\{(a,b,c,-a):a,b,c>0\text{ and }b\geq a,c\geq a\}\backslash
\{(a,b,c,-a):a,b,c>0\text{ and }b>a,c>a\}=
\]%
\[
\left\{  (a,a,c,-a):a,c>0\text{ and }c\geq a\right\}  \cup\left\{
(a,b,a,-a):a,b>0\text{ and }b\geq a\right\}  .
\]
Considering the set $\{(a,a,c,-a):a,c>0\text{ and }c\geq a\}$, we define
\[
h(x,z):=f(x,x,z,-x)=\frac{(3x^{\frac{4}{3}}+z^{\frac{4}{3}})^{\frac{3}{4}}%
}{x+z}.
\]
and
\[
h_{x}=\frac{3x^{\frac{1}{3}}-\Vert(x,x,z,-x)\Vert_{\frac{4}{3}}^{\frac{4}{3}}%
}{\Vert(x,x,z,-x)\Vert_{\frac{4}{3}}^{\frac{1}{3}}(x+z)^{2}},
\]%
\[
h_{z}=\frac{z^{\frac{1}{3}}-\Vert(x,x,z,-x)\Vert_{\frac{4}{3}}^{\frac{4}{3}}%
}{\Vert(x,x,z,-x)\Vert_{\frac{4}{3}}^{\frac{1}{3}}(x+z)^{2}}.
\]
Thus
\[
\nabla h=0\Leftrightarrow(x,z) = \left( \frac{1}{28},\frac{27}{28}\right) 
\]
and hence $h\left( \frac{1}{28},\frac{27}{28}\right) <\sqrt{2}$. Note that
%\[
%f(x,x,9x,-x)=\frac{(3x^{\frac{4}{3}}+(9x)^{\frac{4}{3}})^{\frac{3}{4}}}%
%{10x}<\sqrt{2}.
%\]
%and%
\[
f(x,x,x,-x)=\frac{\Vert(x,x,x,-x)\Vert_{\frac{4}{3}}}{x+x+x+-x}=\frac
{2^{\frac{3}{2}}x}{2x}=\sqrt{2}.
\]
The other cases are similar.

Now we consider the case ($\beta$). Defining
\[
f(x,y,z,w):=\frac{\Vert(x,y,z,w)\Vert_{\frac{4}{3}}}{x+y-z-w},
\]
we have
\[
f_{x}(x,y,z,w)=\frac{x^{\frac{1}{3}}(x+y-z-w)-\Vert(x,y,z,w)\Vert_{\frac{4}%
{3}}^{\frac{4}{3}}}{\Vert(x,y,z,w)\Vert_{\frac{4}{3}}^{\frac{1}{3}%
}(x+y-z-w)^{2}},
\]%
\[
f_{y}(x,y,z,w)=\frac{y^{\frac{1}{3}}(x+y-z-w)-\Vert(x,y,z,w)\Vert_{\frac{4}%
{3}}^{\frac{4}{3}}}{\Vert(x,y,z,w)\Vert_{\frac{4}{3}}^{\frac{1}{3}%
}(x+y-z-w)^{2}},
\]%
\[
f_{z}(x,y,z,w)=\frac{z^{\frac{1}{3}}(x+y-z-w)+\Vert(x,y,z,w)\Vert_{\frac{4}%
{3}}^{\frac{4}{3}}}{\Vert(x,y,z,w)\Vert_{\frac{4}{3}}^{\frac{1}{3}%
}(x+y-z-w)^{2}},
\]
and
\[
f_{w}(x,y,z,w)=\frac{w^{\frac{1}{3}}(x+y-z-w)+\Vert(x,y,z,w)\Vert_{\frac{4}%
{3}}^{\frac{4}{3}}}{\Vert(x,y,z,w)\Vert_{\frac{4}{3}}^{\frac{1}{3}%
}(x+y-z-w)^{2}}.
\]
Since $z,(x+y-z-w),\Vert(x,y,z,w)\Vert_{\frac{4}{3}}>0$, we have
\[
z^{\frac{1}{3}}(x+y-z-w)+\Vert(x,y,z,w)\Vert_{\frac{4}{3}}^{\frac{4}{3}}>0,
\]
and $\nabla f\neq0$. Note that in ($\beta$) we have $c\leq-d$ and, by Lemma
\ref{maxpos}, we have
\begin{align*}
H_{1}  &  :=\{(a,b,c,d):a,b,c>0,d<0,a+b-c-d\geq|a-b|+c-d\}\\
&  =\{(a,b,c,d):a,b,c>0,\ d<0,\ -d\geq c,b\geq c,a\geq c\}.
\end{align*}
So, if $(x,y,z,w)\in A_{1}:=\{(a,b,c,d):a,b,c>0,\ d<0,\ -d>c,b>c,a>c\}$, then
$\nabla f(x,y,z,w)\neq0$, and the maximum of $f$ does not belong to $A$. We
conclude that the maximum of $f$ belongs to
\begin{align*}
H_{1}\backslash A  &  =\{(a,b,a,d):a,b>0,\ d<0\text{ and }b\geq a,-d\geq a\}\\
&  \cup\{(a,b,b,d):a,b>0,\ d<0\text{ and }a\geq b,-d\geq b\}\\
&  \cup\{(a,b,c,-c):a,b,c>0\text{ and }a\geq c,b\geq c\}.
\end{align*}
For $\{(a,b,a,d):a,b>0,\ d<0\text{ and }b\geq a,-d\geq a\}$, we define
\[
g(x,y,w):=\frac{(2x^{\frac{4}{3}}+y^{\frac{4}{3}}+w^{\frac{4}{3}})^{\frac
{3}{4}}}{y-w}%
\]
and
\[
g_{x}=\frac{2x^{\frac{1}{3}}}{\Vert(x,y,x,w)\Vert_{\frac{4}{3}}^{\frac{1}{3}%
}(y-w)}.
\]
We thus conclude that
\[
g_{x}=0\Leftrightarrow x=0
\]
and $\nabla g\neq0$ for all points of $\{(a,b,a,d):a,b>0,\ d<0\text{ and
}b>a,-d>a\}$. Hence, the maximum of $g$ belongs to
\[
\{(a,b,a,d):a,b>0,\ d<0\text{ and }b\geq a,-d\geq a\}\backslash
\{(a,b,a,d):a,b>0,\ d<0\text{ and }b>a,-d>a\}=
\]%
\[
\{(a,a,a,d):a>0,\ d<0\text{ and }-d\geq a\}\cup\{(a,b,a,-a):a,b>0\text{ and
}b\geq a\}.
\]
For the set $\{(a,a,a,d):a>0,\ d<0\text{ and }-d\geq a\}$, we define
\[
h(x,w):=\frac{(3x^{\frac{4}{3}}+w^{\frac{4}{3}})^{\frac{3}{4}}}{x-w}.
\]
Note that
\[
h_{x}=\frac{3x^{\frac{1}{3}}-\Vert(x,x,x,w)\Vert_{\frac{4}{3}}^{\frac{4}{3}}%
}{\Vert(x,x,x,w)\Vert_{\frac{4}{3}}^{\frac{1}{3}}(x-w)^{2}}%
\]
and
\[
h_{w}=\frac{w^{\frac{1}{3}}+\Vert(x,x,x,w)\Vert_{\frac{4}{3}}^{\frac{4}{3}}%
}{\Vert(x,x,x,w)\Vert_{\frac{4}{3}}^{\frac{1}{3}}(x-w)^{2}}.
\]
Thus
\[
3x^{\frac{1}{3}}-\Vert(x,x,x,w)\Vert_{\frac{4}{3}}^{\frac{4}{3}} =0 \text{ and
}w^{\frac{1}{3}}+\Vert(x,x,x,w)\Vert_{\frac{4}{3}}^{\frac{4}{3}}%
=0\Leftrightarrow(x,w) = \left( 0,0\right) .
\]
Hence the maximum of $f$ is obtained in
%\[
%f(x,x,x,-9x)\leq\frac{3+9}{10}=1.2<\sqrt{2}%
%\]
%and%
\[
f(x,x,x,-x)=\sqrt{2}.
\]
The other cases are similar.

The case $\Vert T\Vert=|a_{11}-a_{21}|+a_{12}-a_{22}$ is divided in two cases
\[
(\gamma)\ \Vert T\Vert=a_{11}-a_{21}+a_{12}-a_{22},
\]%
\[
(\delta)\ \Vert T\Vert=-a_{11}+a_{21}+a_{12}-a_{22}.
\]
If ($\gamma$) holds, we consider
\[
f(x,y,z,w)=\frac{\Vert(x,y,z,w)\Vert_{\frac{4}{3}}}{x-y+z-w},
\]
and if ($\delta$) holds we consider
\[
f(x,y,z,w)=\frac{\Vert(x,y,z,w)\Vert_{\frac{4}{3}}}{-x+y+z-w}.
\]
For these cases, using Lemma \ref{maxneg} we prove that the maximum is
attained when $x=y=z=-w$.
\end{proof}

\section{The complex case: numerical and analytical
considerantions\label{complexo}}

It is well known (see \cite{jjj}) that for complex scalars we have%
\[
\left(  \sum\limits_{i,j=1}^{\infty}\left\vert U(e_{i},e_{j})\right\vert
^{\frac{4}{3}}\right)  ^{\frac{3}{4}}\leq\frac{2}{\sqrt{\pi}}\left\Vert
U\right\Vert
\]
for all continuous $m$-linear forms $U:c_{0}\times c_{0}\rightarrow\mathbb{C}%
$, but the it is unknown if the constant $\frac{2}{\sqrt{\pi}}$ is sharp.

Since the optimal constant of the Littlewood's $4/3$ inequality is achieved
when considering simple looking bilinear forms with only four monomials, it is
natural to begin the investigation of the complex case in a similar setting.
We have strong numerical evidence that considering $T:c_{0}\times
c_{0}\rightarrow\mathbb{C}$ given by $T(z,w)=\sum_{i,j=1}^{2}a_{ij}z_{i}w_{j}$
with $a_{ij}\in\mathbb{R}$ the optimal constant is $1$. \ For instance, we can
consider a discretized region $S$ in $[-1,1]^{4}$ such that for any
$T(z,w)=\sum_{i,j=1}^{2}a_{ij}z_{i}w_{j}$ with $a_{ij}\in\lbrack-1,1],$ there
are $s_{1},s_{2},s_{3},s_{4}\in S$ such that
\[
\left\vert a_{11}-s_{1}\right\vert ,\left\vert a_{12}-s_{2}\right\vert
,\left\vert a_{21}-s_{3}\right\vert ,\left\vert a_{22}-s_{4}\right\vert
<10^{-1},
\]
and, for all such $s_{1},s_{2},s_{3},s_{4}$ we have%
\[
\left(  \sum_{i,j=1}^{2}|U(e_{i},e_{j})|^{\frac{4}{3}}\right)  ^{\frac{3}{4}%
}\leq\Vert U\Vert,
\]
for $U(z,w)=s_{1}z_{1}w_{1}+s_{2}z_{1}w_{2}+s_{3}z_{2}w_{1}+s_{4}z_{2}w_{2}.$

The following theorem gives a formal proof that in several cases the optimal
constant is in fact $1:$

\begin{theorem}
\label{caractcomp} Let $T:c_{0}\times c_{0}\rightarrow\mathbb{C}$ be given by
$T(z,w)=\sum_{i,j=1}^{2}a_{ij}z_{i}w_{j}$ with $a_{ij}\in\mathbb{R}$. Then
\[
\left(  \sum_{i,j=1}^{2}|T(e_{i},e_{j})|^{\frac{4}{3}}\right)  ^{\frac{3}{4}%
}\leq\Vert T\Vert,
\]
when

(1) $a_{11}a_{21}=0$ or $a_{12}a_{22}=0$;

(2) $a_{11}a_{21}>0$ and $a_{12}a_{22}>0$;

(3) $a_{11}a_{21}<0$ and $a_{12}a_{22}<0$;

(4) $a_{11}a_{21}>0$ and $a_{12}a_{22}<0$ and $a_{11}a_{21}+a_{12}a_{22}=0$;

(5) $a_{11}a_{21}<0$ and $a_{12}a_{22}>0$ and $a_{11}a_{21}+a_{12}a_{22}=0$.
\end{theorem}

\begin{proof}
For the case (1), there is no loss of generality in supposing $a_{11}=0$. By
Proposition \ref{normcomp}, we have
\[
\Vert T\Vert=\max\{|a_{21}|+|a_{12}+a_{22}|,|a_{21}|+|a_{12}-a_{22}|\}.
\]
If $a_{12}a_{22}\geq0$, then $\Vert T\Vert=|a_{21}|+|a_{12}+a_{22}%
|=\Vert(a_{11},a_{12},a_{21},a_{22})\Vert_{1}$. Now, if $a_{12}a_{22}\leq0$,
then $\Vert T\Vert=|a_{21}|+|a_{12}-a_{22}|=\Vert(a_{11},a_{12},a_{21}%
,a_{22})\Vert_{1}$. Therefore
\[
\frac{\left(  \sum_{i,j=1}^{2}|T(e_{i},e_{j})|^{\frac{4}{3}}\right)
^{\frac{3}{4}}}{\Vert T\Vert}\leq1.
\]

For the case (2), by Proposition \ref{normcomp}, we have
\begin{align*}
\Vert T\Vert &  =\max\{|a_{11}+a_{21}|+|a_{12}+a_{22}|,|a_{11}-a_{21}%
|+|a_{12}-a_{22}|\}\\
&  =|a_{11}+a_{21}|+|a_{12}+a_{22}|=\Vert(a_{11},a_{12},a_{21},a_{22}%
)\Vert_{1},
\end{align*}
and, again,
\[
\frac{\left(  \sum_{i,j=1}^{2}|T(e_{i},e_{j})|^{\frac{4}{3}}\right)
^{\frac{3}{4}}}{\Vert T\Vert}\leq1.
\]

Considering (3), again Proposition \ref{normcomp} tells us that
\begin{align*}
\Vert T\Vert &  =\max\{|a_{11}+a_{21}|+|a_{12}+a_{22}|,|a_{11}-a_{21}%
|+|a_{12}-a_{22}|\}\\
&  =|a_{11}-a_{21}|+|a_{12}-a_{22}|=\Vert(a_{11},a_{12},a_{21},a_{22}%
)\Vert_{1},
\end{align*}
and
\[
\frac{\left(  \sum_{i,j=1}^{2}|T(e_{i},e_{j})|^{\frac{4}{3}}\right)
^{\frac{3}{4}}}{\Vert T\Vert}\leq1.
\]

Now we deal with the case (4). There is no loss of generality in supposing
$a_{11},a_{21},a_{21}>0$ and $a_{22}<0$. Since $a_{12}a_{22}=-a_{11}a_{21}$ we
have
\[
2a_{11}a_{21}\left(  1-\frac{a_{11}a_{21}}{a_{12}a_{22}}\right)
=4a_{11}a_{21}=-4a_{12}a_{22}%
\]
and
\[
\frac{a_{11}^{2}a_{21}^{2}}{a_{12}^{2}a_{22}^{2}}\left(  a_{12}^{2}+a_{22}%
^{2}\right)  -\left(  a_{11}^{2}+a_{21}^{2}\right)  =a_{12}^{2}+a_{22}%
^{2}-(a_{11}^{2}+a_{21}^{2}),
\]
and we obtain
\[
\frac{\frac{a_{11}^{2}a_{21}^{2}}{a_{12}^{2}a_{22}^{2}}\left(  a_{12}%
^{2}+a_{22}^{2}\right)  -\left(  a_{11}^{2}+a_{21}^{2}\right)  }{2a_{11}%
a_{21}\left(  1-\frac{a_{11}a_{21}}{a_{12}a_{22}}\right)  }=\frac{a_{12}%
^{2}+a_{22}^{2}-(a_{11}^{2}+a_{21}^{2})}{4a_{11}a_{21}}.
\]
If
\[
\left\vert \frac{\frac{a_{11}^{2}a_{21}^{2}}{a_{12}^{2}a_{22}^{2}}\left(
a_{12}^{2}+a_{22}^{2}\right)  -\left(  a_{11}^{2}+a_{21}^{2}\right)  }%
{2a_{11}a_{21}\left(  1-\frac{a_{11}a_{21}}{a_{12}a_{22}}\right)  }\right\vert
\leq1,
\]
then
\[
a_{11}^{2}+a_{21}^{2}+2a_{11}a_{21}\frac{a_{12}^{2}+a_{22}^{2}-(a_{11}%
^{2}+a_{21}^{2})}{4a_{11}a_{21}}\geq0
\]
and
\begin{align}
\sqrt{a_{11}^{2}+a_{21}^{2}+2a_{11}a_{21}\frac{a_{12}^{2}+a_{22}^{2}%
-(a_{11}^{2}+a_{21}^{2})}{4a_{11}a_{21}}}  &  =\sqrt{a_{11}^{2}+a_{21}%
^{2}+\frac{a_{12}^{2}+a_{22}^{2}-(a_{11}^{2}+a_{21}^{2})}{2}}
\label{estrela11}\\
&  =\sqrt{\frac{a_{12}^{2}+a_{22}^{2}+a_{11}^{2}+a_{21}^{2}}{2}}\nonumber\\
&  =2^{-\frac{1}{2}}\sqrt{a_{12}^{2}+a_{22}^{2}+a_{11}^{2}+a_{21}^{2}%
}.\nonumber
\end{align}
A similar reasoning tells us that
\begin{equation}
\sqrt{a_{12}^{2}+a_{22}^{2}+2a_{12}a_{22}\frac{a_{12}^{2}+a_{22}^{2}%
-(a_{11}^{2}+a_{21}^{2})}{4a_{11}a_{21}}}=2^{-\frac{1}{2}}\sqrt{a_{12}%
^{2}+a_{22}^{2}+a_{11}^{2}+a_{21}^{2}}. \label{estrela22}%
\end{equation}
Summing up (\ref{estrela11}) and (\ref{estrela22}) we obtain%
\begin{align*}
\sqrt{a_{11}^{2}+a_{21}^{2}+2a_{11}a_{21}\frac{a_{12}^{2}+a_{22}^{2}%
-(a_{11}^{2}+a_{21}^{2})}{4a_{11}a_{21}}}+\sqrt{a_{12}^{2}+a_{22}^{2}%
+2a_{12}a_{22}\frac{a_{12}^{2}+a_{22}^{2}-(a_{11}^{2}+a_{21}^{2})}%
{4a_{11}a_{21}}}  & \\
=2^{\frac{1}{2}}\sqrt{a_{11}^{2}+a_{21}^{2}+a_{12}^{2}+a_{22}^{2}}.  &
\end{align*}
By Proposition \ref{normcomp} we conclude that
\[
\Vert T\Vert=2^{\frac{1}{2}}\sqrt{a_{11}^{2}+a_{21}^{2}+a_{12}^{2}+a_{22}^{2}%
}.
\]
Therefore, by the H\"{o}lder inequality we have
\[
\Vert(a_{11},a_{12},a_{21},a_{22})\Vert_{\frac{4}{3}}\leq2^{\frac{1}{2}}%
\Vert(a_{11},a_{12},a_{21},a_{22})\Vert_{2}=\Vert T\Vert.
\]

Recall that by Lemma \ref{tec01}, we have
\[
\left\vert \frac{\frac{a_{11}^{2}a_{21}^{2}}{a_{12}^{2}a_{22}^{2}}\left(
a_{12}^{2}+a_{22}^{2}\right)  -\left(  a_{11}^{2}+a_{21}^{2}\right)  }%
{2a_{11}a_{21}\left(  1-\frac{a_{11}a_{21}}{a_{12}a_{22}}\right)  }\right\vert
\geq1
\]
if, and only if,
\[
\left\vert \frac{a_{12}a_{22}}{a_{11}a_{21}}(a_{11}-a_{21})\right\vert \geq
a_{12}-a_{22}\text{ or }\left\vert \frac{a_{11}a_{21}}{a_{12}a_{22}}%
(a_{12}+a_{22})\right\vert \geq a_{11}+a_{21}.
\]
Since $a_{11}a_{21}+a_{12}a_{22}=0$, it follows that
\[
(a_{21},a_{22})\in\mathbb{R}(a_{12},-a_{11}),
\]
i.e., $a_{21}=ta_{12}$ and $a_{22}=-ta_{11}$ with $t>0$. Thus, by Proposition
\ref{normcomp} we have
\[
\Vert T\Vert=\max\{a_{11}+ta_{12}+|a_{12}-ta_{11}|,|a_{11}-ta_{12}%
|+a_{12}+ta_{11}\}.
\]
Note that we have two cases:

(a.1) $\|T\| = a_{11}+ta_{12}+|a_{12}-ta_{11}|$;

(a.2) $\Vert T\Vert=|a_{11}-ta_{12}|+a_{12}+ta_{11}$.

In the case (a.1), by Lemma \ref{maxpos}, we have
\begin{equation}
a_{11}\geq-a_{22},\ a_{21}\geq-a_{22},\ a_{12}\geq-a_{22} \label{ppp}%
\end{equation}
or%
\begin{equation}
a_{11}\geq a_{12},\ a_{21}\geq a_{12},\ -a_{22}\geq a_{12}. \label{ppp22}%
\end{equation}
Let us suppose (\ref{ppp})$.$ Note that
\[
a_{11}\geq-a_{22},\ a_{21}\geq-a_{22},\ a_{12}\geq-a_{22}\Leftrightarrow1\geq
t,\ a_{12}\geq a_{11},\ a_{12}\geq ta_{11}.
\]
In this case,
\[
\Vert T\Vert=(1-t)a_{11}+(1+t)a_{12}.
\]
Since $0<t\leq1$, we have
\[
t-1\leq0\Rightarrow(t-1)a_{12}\leq(t+1)a_{11}\Rightarrow-a_{12}-ta_{11}\leq
a_{11}-ta_{12}%
\]
and
\[
1-t\leq1+t\Rightarrow(1-t)a_{11}\leq(1+t)a_{12}\Rightarrow a_{11}-ta_{12}\leq
a_{12}+ta_{11}.
\]
We thus conclude that
\begin{equation}
|a_{11}-ta_{12}|\leq a_{12}+ta_{11}. \label{es}%
\end{equation}
Note that
\begin{equation}
(t-1)a_{11}\leq(1+t)a_{12}\Rightarrow-a_{11}-ta_{12}\leq a_{12}-ta_{11}.
\label{es2}%
\end{equation}
and also that
\begin{equation}
a_{12}-ta_{11}\leq a_{11}+ta_{12}\Leftrightarrow a_{12}-a_{11}\leq
t(a_{11}+a_{12})\Leftrightarrow\frac{a_{12}-a_{11}}{a_{11}+a_{12}}\leq t.
\label{es3}%
\end{equation}
If $a_{11}=a_{12}$, then combining the information of (\ref{es2}) and
(\ref{es3}) we have
\[
|a_{12}-ta_{11}|\leq a_{11}+ta_{12}%
\]
and by Lemma \ref{tec01} we conclude that
\begin{equation}
\left\vert \frac{\frac{a_{11}^{2}a_{21}^{2}}{a_{12}^{2}a_{22}^{2}}\left(
a_{12}^{2}+a_{22}^{2}\right)  -\left(  a_{11}^{2}+a_{21}^{2}\right)  }%
{2a_{11}a_{21}\left(  1-\frac{a_{11}a_{21}}{a_{12}a_{22}}\right)  }\right\vert
\leq1. \label{u78}%
\end{equation}
By (\ref{u78}), using the previous estimates the proof is completed for this case.

Now suppose
\[
a_{11}<a_{12}\text{ and }t\leq\frac{a_{12}-a_{11}}{a_{11}+a_{12}}.
\]
By (\ref{es3}), we have
\[
a_{11}+ta_{12}\leq a_{12}-ta_{11}.
\]
In this case, fixing $a_{11},a_{12}>0$, we define $f:[0,\frac{a_{12}-a_{11}%
}{a_{11}+a_{12}}]\rightarrow\mathbb{R}$ by
\[
f(t)=\frac{\left(  (1+t^{\frac{4}{4}})(a_{11}^{\frac{4}{3}}+a_{12}^{\frac
{4}{3}})\right)  ^{\frac{3}{4}}}{\Vert T\Vert}.
\]
Since
\[
f(0)=\frac{\left(  a_{11}^{\frac{4}{3}}+a_{12}^{\frac{4}{3}}\right)
^{\frac{3}{4}}}{a_{11}+a_{12}}\leq1
\]
and
\[
f\left(  \frac{a_{12}-a_{11}}{a_{11}+a_{12}}\right)  \leq1,
\]
we have
\[
f^{\prime}(t)=\frac{t^{\frac{1}{3}}(a_{11}^{\frac{4}{3}}+a_{12}^{\frac{4}{3}%
})((1-t)a_{11}+(1+t)a_{12})-\left(  (1+t^{\frac{4}{4}})(a_{11}^{\frac{4}{3}%
}+a_{12}^{\frac{4}{3}})\right)  (a_{12}-a_{11})}{\left(  (1+t^{\frac{4}{4}%
})(a_{11}^{\frac{4}{3}}+a_{12}^{\frac{4}{3}})\right)  ^{\frac{1}{4}%
}((1-t)a_{11}+(1+t)a_{12})^{2}}.
\]
Thus, $f^{\prime}(t_{0})=0$ if, and only if,
\[
t_{0}=-\frac{(a_{11}-a_{12})^{3}}{(a_{11}+a_{12})^{3}}.
\]
But $t_{0}$ is a point of minimum for $f$ and thus $\max f\leq1$.

Now we consider the case (\ref{ppp22}).

Note that
\[
a_{11}\geq a_{12},\ a_{21}\geq a_{12},\ -a_{22}\geq a_{12}\Leftrightarrow
a_{11}\geq a_{12},\ t\geq1,\ ta_{11}\geq a_{12}%
\]
and in this case
\[
\Vert T\Vert=(1+t)a_{11}+(t-1)a_{12}.
\]
Since $1\leq t$, we have
\begin{equation}
(t-1)\leq(t+1)\Rightarrow(t-1)a_{11}\leq(t+1)a_{12}\Rightarrow-a_{11}%
-ta_{12}\leq a_{12}-ta_{11}. \label{xis}%
\end{equation}
and
\begin{equation}
1-t\leq0\Rightarrow(1-t)a_{12}\leq(1+t)a_{11}\Rightarrow a_{12}-ta_{11}\leq
a_{11}+ta_{12}. \label{xis22}%
\end{equation}
Thus, by (\ref{xis}) and (\ref{xis22}), we have
\[
|a_{12}-ta_{11}|\leq a_{11}+ta_{12}.
\]
Note also that
\begin{equation}
(1-t)a_{11}\leq(1+t)a_{12}\Rightarrow a_{11}-ta_{12}\leq a_{12}+ta_{11}
\label{qaz}%
\end{equation}
and
\begin{equation}
-a_{12}-ta_{11}\leq a_{11}-ta_{12}\Leftrightarrow-a_{11}-a_{12}\leq
t(a_{12}-a_{11}). \label{qazx}%
\end{equation}
If $a_{12}=a_{11}$, by (\ref{xis}) and (\ref{xis22}), we have
\[
|a_{11}-ta_{12}|\leq a_{12}+ta_{11}%
\]
and thus
\[
\left\vert \frac{\frac{a_{11}^{2}a_{21}^{2}}{a_{12}^{2}a_{22}^{2}}\left(
a_{12}^{2}+a_{22}^{2}\right)  -\left(  a_{11}^{2}+a_{21}^{2}\right)  }%
{2a_{11}a_{21}\left(  1-\frac{a_{11}a_{21}}{a_{12}a_{22}}\right)  }\right\vert
\leq1
\]
and the proof of this case is done. If $a_{12}<a_{11}$ and $\frac
{a_{11}+a_{12}}{a_{11}-a_{12}}\leq t$ then, by (\ref{qazx}) we have
\[
a_{11}-ta_{12}\leq-a_{12}-ta_{11}.
\]
Fixing $a_{11},a_{12}>0$, we consider $g:\left[  \frac{a_{11}+a_{12}}%
{a_{11}-a_{12}},\infty\right)  \rightarrow\mathbb{R}$ by
\[
g(t)=\frac{\left(  (1+t^{\frac{4}{3}})(a_{11}^{\frac{4}{3}}+a_{12}^{\frac
{4}{3}})\right)  ^{\frac{3}{4}}}{\Vert T\Vert}.
\]
Note that
\[
g\left(  \frac{a_{11}+a_{12}}{a_{11}-a_{12}}\right)  \leq1
\]
and
\[
g(t)=\frac{t\left(  (a_{11}^{\frac{4}{3}}+a_{12}^{\frac{4}{3}})+\frac
{(a_{11}^{\frac{4}{3}}+a_{12}^{\frac{4}{3}})}{t^{\frac{4}{3}}}\right)
^{\frac{3}{4}}}{t((a_{11}+a_{12})+\frac{a_{11}-a_{12}}{t})}=\frac{\left(
(a_{11}^{\frac{4}{3}}+a_{12}^{\frac{4}{3}})+\frac{(a_{11}^{\frac{4}{3}}%
+a_{12}^{\frac{4}{3}})}{t^{\frac{4}{3}}}\right)  ^{\frac{3}{4}}}%
{((a_{11}+a_{12})+\frac{a_{11}-a_{12}}{t})},
\]
and
\[
\displaystyle\lim_{t\rightarrow\infty}g(t)=\frac{\left(  (a_{11}^{\frac{4}{3}%
}+a_{12}^{\frac{4}{3}})\right)  ^{\frac{3}{4}}}{a_{11}+a_{12}}\leq1.
\]
Moreover,
\[
g^{\prime}(t)=\frac{t^{\frac{1}{3}}(a_{11}^{\frac{4}{3}}+a_{12}^{\frac{4}{3}%
})((1+t)a_{11}+(t-1)a_{12})-\left(  (1+t^{\frac{4}{4}})(a_{11}^{\frac{4}{3}%
}+a_{12}^{\frac{4}{3}})\right)  (a_{11}+a_{12})}{\left(  (1+t^{\frac{4}{4}%
})(a_{11}^{\frac{4}{3}}+a_{12}^{\frac{4}{3}})\right)  ^{\frac{1}{4}%
}((1+t)a_{11}+(t-1)a_{12})^{2}}.
\]
Therefore $g^{\prime}(t_{0})=0$ if, and only if,
\[
t_{0}=\frac{(a_{11}+a_{12})^{3}}{(a_{11}-a_{12})^{3}},
\]
and $t_{0}$ is a point of minimum for $g$ and we conclude that $\max g\leq1$.

The case (a.2) is similar and (5) is analogous to (4).
\end{proof}

\begin{corollary}
Let $\alpha\neq0$. Then $T:\ell_{\infty}^{2}(\mathbb{C})\times\ell_{\infty
}^{2}(\mathbb{C})\rightarrow\mathbb{C}$ given by
\[
T(x,y)=\alpha z_{1}w_{1}+\alpha z_{1}w_{2}+\alpha z_{2}w_{1}-\alpha z_{2}w_{2}%
\]
or
\[
T(x,y)=\alpha z_{1}w_{1}+\alpha z_{1}w_{2}-\alpha z_{2}w_{1}+\alpha z_{2}w_{2}%
\]
or
\[
T(x,y)=\alpha z_{1}w_{1}-\alpha z_{1}w_{2}+\alpha z_{2}w_{1}+\alpha z_{2}w_{2}%
\]
or
\[
T(x,y)=-\alpha z_{1}w_{1}+\alpha z_{1}w_{2}+\alpha z_{2}w_{1}+\alpha
z_{2}w_{2}%
\]
satisfies
\[
\left(  \displaystyle\sum_{i,j=1}^{2}|T(e_{i},e_{j})|^{\frac{4}{3}}\right)
^{\frac{3}{4}}=\Vert T\Vert.
\]

\end{corollary}

\section{Appendix: some elementary lemmata}

This section is entirely devoted to five elementary lemmata used in the paper.
The first two lemmata are quite simple and we omit their proof.

\begin{lemma}
\label{maxig}Let $a,b,c>0$ and $d<0$ be such that
\[
|a+b|+|c+d|=|a-b|+|c-d|.
\]
If $a\neq b$, then $b=-d$ or $a=-d$.
\end{lemma}

\begin{lemma}
\label{888000}Let $a,b\in\mathbb{R}\backslash\{0\}$.

(a) If $|t_{0}|\leq1$, then $|a+bt_{0}|\leq|a+b|$ or $|a+bt_{0}|\leq|a-b|$.

(b) The maximum of the function $f:[-1,1]$ $\rightarrow\mathbb{R}$ given by
\[
f(t)=|a+bt|+|c+dt|
\]
occurs for $t=-1$ or $t=1$.
\end{lemma}

\begin{lemma}
\label{maxpos}Let $a,b,c>0$ and $d<0$. Then
\[
a+b+|c+d|\geq|a-b|+c-d
\]
if, and only if,
\[
a\geq-d,\ b\geq-d,\ c\geq-d\text{ or }a\geq c,\ b\geq c,\ -d\geq c.
\]

\end{lemma}

\begin{proof}
Suppose that $a,b,c>0$ and $d<0$ are such that
\[
a+b+|c+d|\geq|a-b|+c-d.
\]
If $c\geq-d$, we have
\[
a+b+c+d\geq|a-b|+c-d.
\]
Since $a-b\leq|a-b|$ and $b-a\leq|b-a|=|a-b|$, we have
\[
a+b+c+d\geq a-b+c-d
\]
and thus $b\geq-d.$ We also have
\[
a+b+c+d\geq b-a+c-d
\]
and thus $a\geq-d.$

Now suppose that $-d\geq c.$ Then
\[
a+b-d-c\geq a-b+c-d
\]
and hence $b\geq c$.\ Besides,
\[
a+b-d-c\geq b-a+c-d
\]
and thus $a\geq c.$

Reciprocally, suppose that $a,b,c>0$ and $d<0$ with $a\geq-d,\ b\geq
-d,\ c\geq-d$. Then
\[
b\geq-d\Rightarrow2b\geq-2d\Rightarrow b+d\geq-b-d\Rightarrow a+b+c+d\geq
a-b+c-d
\]
and
\[
a\geq-d\Rightarrow2a\geq-2d\Rightarrow a+d\geq-a-d\Rightarrow a+b+c+d\geq
b-a+c-d.
\]
Hence
\[
a+b+c+d\geq|a-b|+c-d.
\]

Now, suppose that $a,b,c>0$ and $d<0,$ with $a\geq c,\ b\geq c,\ -d\geq c$.
Then
\[
b\geq c\Rightarrow2b\geq2c\Rightarrow b-c\geq-b+c\Rightarrow a+b-c-d\geq
a-b+c-d
\]
and
\[
a\geq c\Rightarrow2a\geq2c\Rightarrow a-c\geq-a+c\Rightarrow a+b-c-d\geq
b-a+c-d.
\]
and we finally obtain
\[
a+b-c-d\geq|a-b|+c-d.
\]

\end{proof}

The next lemma has a proof similar to the proof of the previous lemma, and we
omit it.

\begin{lemma}
\label{maxneg}Let $a,b,c>0$ and $d<0$. Then
\[
|a-b|+c-d\geq a+b+|c+d|
\]
if, and only if,
\[
a\geq b,c\geq b,-d\geq b\text{ or }b\geq a,c\geq a,-d\geq a.
\]

\end{lemma}

\begin{lemma}
\label{tec01} Let $a,b,c>0$ and $d<0$. Then
\[
\left\vert \left(  \frac{ab}{cd}\right)  ^{2}(c^{2}+d^{2})-(a^{2}%
+b^{2})\right\vert \leq2ab\left(  1-\frac{ad}{cd}\right)
\]
if, and only if,
\[
\left\vert \frac{cd}{ab}(a-b)\right\vert \leq c-d\text{ and }\left\vert
\frac{ab}{cd}(c+d)\right\vert \leq a+b.
\]

\end{lemma}

\begin{proof}
Note that%
\[
\left\vert \left(  \frac{ab}{cd}\right)  ^{2}(c^{2}+d^{2})-(a^{2}%
+b^{2})\right\vert \leq2ab\left(  1-\frac{ad}{cd}\right)
\]
if, and only if,
\[
-2ab\left(  1-\frac{ad}{cd}\right)  \leq\left(  \frac{ab}{cd}\right)
^{2}(c^{2}+d^{2})-(a^{2}+b^{2})\leq2ab\left(  1-\frac{ad}{cd}\right)  .
\]
Moreover
\[
-2ab\left(  1-\frac{ad}{cd}\right)  \leq\left(  \frac{ab}{cd}\right)
^{2}(c^{2}+d^{2})-(a^{2}+b^{2})
\]
if, and only if
\[
\left\vert \frac{cd}{ab}(a-b)\right\vert \leq c-d.
\]
In a similar fashion we show that the other inequality is equivalent to
\[
\left\vert \frac{ab}{cd}(c+d)\right\vert \leq a+b.
\]

\end{proof}

\end{document}